\documentclass{scrartcl}

\usepackage{graphicx} 
\usepackage[T1]{fontenc}
\usepackage[utf8]{inputenc}
\usepackage[english]{babel}
\usepackage{enumitem}
\usepackage{amsmath}
\usepackage{amsthm}
\usepackage{url}
\usepackage{bbm}
\usepackage{hyperref}
\usepackage{amssymb}
\usepackage{tikz}
\usepackage{graphicx}
\usepackage{faktor}
\usepackage{xcolor, colortbl}
\usepackage{makecell}
\usepackage{caption}
\usepackage{subcaption}
\usepackage{comment}
\usepackage{setspace}
\usepackage[12pt]{moresize}
\usepackage{wasysym}
\usepackage{float}
\usepackage[margin=1in]{geometry}
\usepackage{mathrsfs}
\usepackage{url}
\usepackage{authblk}
\usepackage{esint}
\usepackage{bbm}
\usepackage{hyperref}
\usepackage{cleveref}
\hypersetup{
colorlinks = true,
urlcolor   = blue,
citecolor  = blue,
linkcolor  = blue,
}

\usepackage[normalem]{ulem}

\DeclareMathOperator*{\esssup}{ess\,sup}

\newcommand{\R}{\mathbb{R}}

\renewcommand{\d}{\mathrm{d}}

\theoremstyle{plain}
\newtheorem{thm}{Theorem}[section]

\newtheorem*{thm*}{Theorem}

\newtheorem{prop}[thm]{Proposition}

\newtheorem{assumption}{Assumption}

\newtheorem{defn}[thm]{Definition}
\theoremstyle{remark}
\newtheorem{rmk}[thm]{Remark}

\title{A Dacorogna-Moser construction of transport maps on $\R^d$ with application to geodesics on the space of couplings}
\author{Louis-Pierre \textsc{Chaintron}, Matteo \textsc{Picco} \\
\normalsize{\'Ecole Polytechnique Fédérale de Lausanne (EPFL), Institute of Mathematics, 1015
Lausanne, Switzerland.
} \vspace{-0.4cm}}

\date{}

\begin{document}

\maketitle

\begin{abstract}

\begin{center}
\large\textbf{Abstract}
\end{center}
\vspace{0.5em}
A seminal work by Dacorogna and Moser introduced a way of constructing regular transport maps from a probability distribution on a bounded domain to another one. 
In this work, we extend this construction to the whole $\R^d$ for strictly asymptotically log-concave measures, a wide class of distributions that encompasses Lipschitz-perturbations of log-concave measures. A key feature of our extension is to produce transport maps that satisfy {global} regularity bounds.
We then leverage this construction to study geodesics in the space of probability measures on a product set with imposed marginal laws (couplings), for which we derive optimality conditions, answering an open question in a recent work by Conforti, Lacker and Pal. 
Taking inspiration from Brenier's variational model for incompressible fluids and its regularization, we further introduce an entropic regularization of the geodesic problem, which can be seen as the Schrödinger bridge problem on the space of couplings, for which we also derive optimality conditions.
We eventually study convergence of minimizers as the regularization vanishes, and we prove convergence of the related Lagrange multipliers.
Our approach involves proving uniform-in-time global regularity estimates on elliptic and parabolic equations on $\R^d$, by exploiting the structure of asymptotically log-concave measures using the probabilistic notion of reflection coupling.
\end{abstract}

\section{Introduction} \label{sec:intro}

The construction of regular transport maps has attracted significant attention in recent years, motivated by both theoretical and applied considerations. On the theoretical side, such maps play a central role in extending functional inequalities, including Poincaré and log-Sobolev inequalities, to broader classes of probability measures \cite{KimMilmanHeatFlow,LipschitzPropAlongHeatFlow,TransportationLogLipschitz}. From a computational perspective, regular transport maps provide an effective framework for sampling from complex probability distributions \cite{SamplingMeasTransp,chewi2023log}. Indeed, a Lipschitz transport map from a simple reference measure -- such as a Gaussian -- to a target measure enables sampling from the latter while providing explicit error estimates in terms of the Lipschitz constant of the map.
This principle is fundamental in modern generative machine learning, where the objective is to efficiently generate samples from prescribed probability distributions. 
Many approaches have been developed to enlarge the class of target measures that can be handled and to improve the efficiency of sampling, including diffusion models \cite{Ho2020,Song2021}, flow matching \cite{FlowMatching} and stochastic interpolants \cite{StochInterp}. We refer the reader to Subsection \ref{Literature review} for a more detailed overview of these methods and their connections to regular transport maps.

The seminal work \cite{DacorognaMoser} by Dacorogna and Moser constructed Hölder-continuous transport maps between probability densities $f$ and $g$ on a bounded domain $\Omega \subset \R^d$, under mild regularity assumptions on $\partial \Omega$, $f$ and $g$. 
In the particular case where $f$ is the uniform probability density on $\Omega$ and the target is $g = (1+ \varphi)\cdot f$ with $\int_{\Omega}\varphi (x) f (x) \d x = 0$, the transport map in \cite[Lemma 3]{DacorognaMoser} is built from a flow $(\rho_s)_{s \in [0,1]}$ of probability measures that is a weak solution of the transport equation
\begin{equation} \label{eq:transportDM}
 \partial_s \rho_s + \nabla \cdot(\rho_s v_s) = 0, \quad \rho_0 = f \d x, \quad \rho_1 =g \d x, 
\end{equation} for a suitably chosen regular vector field $(v_s)_{s \in [0,1]}$.
The transport map $T : \R^d \rightarrow \R^d$ satisfying 
\[ \rho_1 = T_{\#}\rho_0 \]
is then given by the flow generated by $(v_s)_{s \in [0,1]}$ at time $s =1$. 
A key simplifying assumption for choosing $(v_s)_{s \in [0,1]}$ is
\begin{equation*}
    \rho_s = (1+ s \varphi)\cdot f, \qquad
    v_s = -\frac{\nabla u}{1+s\varphi}, 
\end{equation*} 
which reduces \eqref{eq:transportDM} to solving the system \begin{equation}
\begin{cases} \label{Poisson system DM}
\Delta u = \varphi, &\text{   in   } \Omega, \\
\frac{\partial u}{\partial\nu} = 0, &\text{   on   } \partial\Omega,
\end{cases}
\end{equation}
the choice of the Neumann boundary condition ensuring that the resulting transport map leaves the boundary $\partial\Omega$ invariant \cite[Theorem 1]{DacorognaMoser}.
The desired regularity for the transport map then results from proving estimates on the partial differential equation (PDE) \eqref{Poisson system DM}.
However, the estimates from \cite[Theorem 1]{DacorognaMoser} explicitly depend on the boundedness of $\Omega$, which prevents a direct extension of the construction to the whole $\R^d$.

In this article, we propose an extension of the Dacorogna-Moser construction of transport maps, which is stated in Theorem \ref{Theorem Dacorogna-Moser} below, starting from a probability measure $\mu$ on $\R^d$.
Adapting the strategy from \cite{DacorognaMoser}, we reduce the problem of constructing a transport map from $\mu$ to $(1+\varphi)\cdot\mu$ to solving
\begin{equation} \label{first time new PDE}
\Delta\psi + \nabla\log\mu \cdot \nabla\psi = \varphi, \qquad \int_{\R^d}\psi\d\mu = 0, \qquad \text{in  } \R^d,
\end{equation} 
in place of \eqref{Poisson system DM}.
The regularity of the transport map then follows from proving uniform regularity estimates on the solution of \eqref{first time new PDE}.
To do so, we need to assume some decay for the tails of the measure $\mu$.

Throughout this work, our key assumption is the \emph{strict asymptotic log-concavity} of the reference measure $\mu$. This assumption allows for non-log-concave measures, including Lipschitz perturbations of sufficiently coercive log-concave measures.
To introduce it, we define the \emph{weak semiconvexity profile} of a function $f \in C^1(\R^d,\R)$ as
\begin{equation} \label{weak semiconvexity profile}
\kappa_f(r):= \frac{1}{r^2}\inf_{|x-y|=r}\langle \nabla f (x) - \nabla f (y), x-y\rangle, \quad r > 0.
\end{equation} 
It is classical that $\kappa_f \geq \alpha$, for $\alpha \in \R$, is equivalent to $f$ being $\alpha$-convex, i.e. $x \mapsto f(x) - \frac{\alpha}{2} \vert x \vert^2$ being convex.
The following definition is \cite[Definition 1.1]{WeakConc}.

\begin{defn}[Asymptotic convexity] \label{def asymptotic convexity}
We say that a function $f \in C^1(\R^d,\R)$ is asymptotically convex if \begin{equation*}
       \liminf_{r \to +\infty}\kappa_f(r) \geq 0.
\end{equation*} 
We say that $f$ is \emph{strictly} asymptotically convex if the previous inequality is strict.
\end{defn} 

Similarly, the reference measure $\mu$ is (strictly) asymptotically log-concave if $-\log \mu$ is (strictly) asymptotically convex, underlying that $\mu$ is absolutely continuous with respect to Lebesgue and identifying $\mu$ with its density.
We notice that the class of asymptotically log-concave probability distributions encompasses Lipschitz-perturbations of log-concave measures.

The weak semiconvexity profile was introduced by \cite{IntroSemiConvProfile} together with the probabilistic notion of \emph{coupling by reflection}, and further used in \cite{Eberle} to establish exponential contraction for diffusion operators. 
More recently, the propagation of weak semiconvexity along generalized heat flows was systematically studied in \cite{WeakConc}, which further deduced many useful consequences -- including building Lipschitz transport maps from the Gaussian towards strictly asymptotically log-concave measures \cite[Theorem 1.4]{WeakConc}.
A partial analog of these results had been established in \cite[Theorem 4.1]{SpectralGap} by PDE methods to prove the fundamental gap conjecture.

Exploiting the strict asymptotic log-concavity of the reference measure, a key result of the present article is the following extension of the Dacorogna-Moser construction to the entire $\R^d$.

\begin{thm}[Dacorogna-Moser construction on $\R^d$] \label{Theorem Dacorogna-Moser}
Assume that $\mu$ is a strictly asymptotically log-concave measure on $\R^d$ such that $\nabla^j\log\mu$ is Lipschitz-continuous for every $1\le j \le n+1$, $n \geq 1$. 
Let $\varphi \in W^{n,\infty}(\R^d,\R)$ be such that $\int_{\R^d}\varphi\d\mu = 0$, $\lVert \varphi\rVert_{W^{n,\infty}} \le M$ and $\varphi \ge -1 + \delta$, for some $M, \delta > 0$. 
Then, there exists an invertible $T : \R^d \rightarrow \R^d$ such that  $(1+\varphi)\cdot\mu = T_{\#}\mu$, $T - \mathrm{Id}, \ T^{-1} - \mathrm{Id} \in W^{n,\infty}(\R^d, \R^d)$ and 
\begin{equation*}
    \lVert T - \mathrm{Id} \rVert_{W^{n,\infty}} + \lVert T^{-1} - \mathrm{Id} \rVert_{W^{n,\infty}} \le C \lVert \varphi \rVert_{W^{n,\infty}},
\end{equation*} where $C>0$ only depends on the bounds on $\log \mu$, the weak semiconvexity profile $\kappa_{-\log\mu}, \ n$ and the constants $M$, $\delta$.
\end{thm} 

A key feature of the above result is that the regularity bounds satisfied by $T$ are global in space.
The proof can be found in Subsection \ref{Proofs gen DM}. This result leverages the estimates on the PDE \eqref{first time new PDE} that are proved
in Proposition \ref{Lemma PDE estimates}. In particular, we prove that the unique solution $\psi \in H^1(\mu)$ of \eqref{first time new PDE} satisfies $\lVert \nabla\psi\rVert_{W^{n,\infty}} \leq C \lVert \varphi \rVert_{W^{n,\infty}}$. 
This bound further relies on short- and long-time regularization estimates for the semi-group generated by the operator $\nabla\log\mu \cdot \nabla + \Delta$, see Proposition \ref{Lemma PDE estimates}. 

\begin{rmk}[Additional regularity] \label{rem:addREg}
Comparing to the original  Dacorogna-Moser result \cite[Theorem 1]{DacorognaMoser}, we would expect $T$ to have one more bounded derivative than $\varphi$.
The Schauder regularity theory indeed provides \emph{local} bounds for second-order derivatives of solutions of elliptic equations, which could be leveraged to prove additional \emph{local} regularity for $T$.
On the contrary, the elliptic regularity estimate proved in Proposition \ref{Lemma PDE estimates} is \emph{global}, as well as the result of Theorem \ref{Theorem Dacorogna-Moser}, which is a key consequence of the asymptotic log-concavity of $\mu$.
We do not expect higher-order \emph{global} uniform estimates to be true in general -- except for fractional regularity.
\end{rmk}

Since \cite{DacorognaMoser},
the original Dacorogna-Moser construction has been used in many important works including Brenier's variational model for incompressible fluids \cite{Brenier1999MinimalGO}, where it plays a key role in building competitors to perform calculus of variation arguments. 
Adapting the approach from \cite{Brenier1999MinimalGO}, a key application of Theorem \ref{Theorem Dacorogna-Moser} is to compute variations and derive optimality conditions for the problem of geodesics in the space of couplings of probability measures, which we introduce below, the precise statements of our results being presented later in Section \ref{Statements geodesics}. 
Once again, a key assumption will be the strict asymptotic log-concavity of the coupled marginals $\mu$ and $\nu$.

\subsection{Geodesics in the space of couplings} \label{Subsection intro geodesics}

Let us fix $\mu, \nu$ in the Wasserstein space $\mathcal{P}_2(\R^d)$ of probability measures with finite second moment.
Let us further consider $\rho_0, \rho_1 \in \mathcal{P}_2 ( \R^d \times \R^d)$.
Recalling the Benamou-Brenier formula \cite{BenamouBrenier}, the Wasserstein distance between $\rho_0$ and $\rho_1$ can be written as
\begin{equation} \label{eq:BB}
\mathcal{W}_2^2(\rho_0,\rho_1) = \inf\bigg\{ \int_0^1 \frac12 \|v_t\|_{L^2(\rho_t)}^2\,\d t : \partial_t \rho + \nabla \cdot (\rho v)=0, \, (\rho \vert_{t=0},\rho \vert_{t=1}) = (\rho_0,\rho_1) \bigg\},
\end{equation}
the continuity equation being understood in the weak sense and the factor $\frac12$ being here for later convenience.
This formula allows for viewing $\mathcal{P}_2 ( \R^d \times \R^d)$ as a formal infinite-dimensional Riemannian manifold, as developed in \cite{otto2001geometry}.

Let us now assume that $\rho_0, \rho_1$ belong to the sub-space $\Pi(\mu,\nu) \subset \mathcal{P}_2 ( \R^d \times \R^d)$ of measures with first marginal $\mu$ and second marginal $\nu$. 
If we restrict the minimization \eqref{eq:BB} to curves that stay in $\Pi(\mu,\nu)$ at all times, then we get a formal notion of submanifold distance on $\Pi(\mu,\nu)$ given by 
\begin{equation} \label{def distance}
d_{\Pi(\mu,\nu)}^2 (\rho_0,\rho_1) = \inf\bigg\{ \int_0^1 \frac12 \|v_t\|_{L^2(\rho_t)}^2\,\d t : \partial_t \rho + \nabla \cdot (\rho v)=0, \ (\rho \vert_{t=0},\rho \vert_{t=1}) = (\rho_0,\rho_1), \rho_t \in \Pi(\mu,\nu) \bigg\}.
\end{equation}
We say that $(\rho_t)_{t \in [0,1]}$ is a \emph{geodesic in $\Pi(\mu,\nu)$} between $\rho_0$ and $\rho_1$ if it is a minimizer in \eqref{def distance}.

The problem of geodesics in the space of couplings was first introduced in \cite[Section 2.6]{ConfLackPal}, which was motivated by the construction of gradient flows in the induced geometry of $\Pi(\mu,\nu)$.
In particular, \cite{ConfLackPal} constructed dynamics that converge exponentially fast to the unique solution of the Schr{\"o}dinger problem between $\mu$ and $\nu$, while staying in $\Pi(\mu,\nu)$ at every time when started from there. 
Although the long-time behavior of these dynamics shares some similarities with gradient flows in $(\Pi(\mu,\nu),d_{\Pi(\mu,\nu)})$, they do not satisfy the rigorous requirements for being gradient flows.
A major difficulty to build gradient flows is indeed the lack of geodesic convexity of $\Pi(\mu,\nu)$ observed in \cite[Proposition 2.2]{ConfLackPal}, so that the standard theory of gradient flows \cite{ambrosioGradientFlowsMetric2008} -- which usually requires geodesic convexity for the functional generating the flow -- cannot be directly applied. 
This motivated the open question \cite[Open Problem 6]{ConfLackPal} of describing geodesics in $\Pi(\mu,\nu)$.
The existence of geodesics was studied in \cite[Proposition 3.1, Proposition 3.6]{LouisLacker} under the assumption that $\mu$, $\nu$ are strongly log-concave and that $\log \mu$, $\log \nu$ have bounded Hessian.
In Section \ref{Statements geodesics} below, we extend this existence result to marginals that are merely strictly asymptotically log-concave, before giving a complete PDE description of the geodesics answering \cite[Open Problem 6]{ConfLackPal}.

By studying the optimality conditions for \eqref{def distance}, we prove in Theorem \ref{Existence of the Lagrange multiplier} existence for Lagrange multipliers $\alpha$ and $\beta$ associated to the marginal constraints, which formally satisfy
\begin{equation} 
    \begin{cases}\label{formal optimality conditions}
        \partial_t \rho + \nabla \cdot (\rho \nabla \varphi) = 0\\
        \alpha \oplus \beta = \partial_t \varphi + \frac{1}{2}|\nabla\varphi|^2,
    \end{cases}
\end{equation} where $\varphi$ is a multiplier associated to the continuity equation. Differentiating the second equation and combining it with the first one, we can formally write the conservation of momentum
\begin{equation} \label{weaker opt cond}
    \rho \nabla(\alpha \oplus \beta) = \partial_t(\rho \nabla\varphi) + \nabla \cdot (\rho(\nabla\varphi \otimes \nabla\varphi)).
\end{equation}
The above computations are only formal, as we assumed existence and smoothness for the Lagrange multipliers $\alpha$ and $\beta$, which does not hold in general. 
Under strict asymptotic log-concavity and regularity assumptions on the marginals, Theorem \ref{Existence of the Lagrange multiplier} below gives a rigorous meaning to $\alpha$ and $\beta$ as elements in the dual of some suitable functional spaces, which are rigorously defined in Definition \ref{Perturbations and modified problem}.
Under additional regularity assumptions, Theorem \ref{Optimality condition for the Lagrange multiplier} then gives a rigorous meaning to the optimality conditions \eqref{weaker opt cond} for $\alpha$ and $\beta$, deducing uniqueness for the multipliers at the same time. 
These results rely on calculus of variation arguments that crucially leverage our extended Dacorogna-Moser construction Theorem \ref{Theorem Dacorogna-Moser} to build competitors, similarly to how Brenier's work \cite{Brenier1999MinimalGO} used the original bounded-domain Dacorogna-Moser construction.

The geometric structure of this problem is reminiscent of Brenier's multiphase formulation of Euler's incompressible equations \cite{brenier1989least,Brenier1999MinimalGO}, which is itself a relaxation of Arnold's formulation \cite{Arnold} of Euler's equations for incompressible fluids.
We recall that Arnold's variational problem
is given by 
\begin{equation*}
    \inf_{(X_t)_{t \in [0,1]}} \int_0^1 \int_{\mathbb{T}^d} \frac12|\dot X_t(a)|^2\d a \d t,
\end{equation*} where the infimum is taken over absolutely continuous curves of diffeomorphisms $X_t : \mathbb{T}^d \rightarrow \mathbb{T}^d$ of the torus that preserve orientation and the Lebesgue measure (incompressibility constraint) at any time, with fixed endpoints at $t=0$ and $t=1$.
A pressure field can then be introduced as a Lagrange multiplier for the incompressibility constraint, and the formal optimality conditions then correspond to a Lagrangian formulation of Euler's equations.
This problem is by now known as \emph{incompressible optimal transport}, and
we refer to \cite{ambrosio2010lecture,daneri2012variational} for lecture notes on these variational models as well as to the textbook \cite[Chapters 2-4]{brenier2020examples}.

Arnold's problem is notoriously difficult as minimization is performed over a set that is not closed nor convex, and a minimizer does not exist in general \cite{SDiff87,SDiff94}.
Therefore, \cite{Brenier1999MinimalGO} introduced the convex multiphase relaxation \begin{equation} \label{eq:Bremutli}
    \inf_{(\rho^a,v^a)_{a \in \mathbb{T}^d}}\bigg\{ \int_{\mathbb{T}^d} \int_0^1\int_{\mathbb{T}^d} \frac12| v^a_t |^2 \d \rho^a_t \d t \d a, \; \ \int_{\mathbb{T}^d} \rho_t^a \d a = \mathcal{L}_{\mathbb{T}^d} \  \forall t \in [0,1]\bigg\},
\end{equation} 
where $(\rho^a,v^a)$ satisfies the continuity equation for almost every $a \in \mathbb{T}^d$, $\mathcal{L}_{\mathbb{T}^d}$ is the Lebesgue measure on $\mathbb{T}^d$ and the endpoints at $t=0$ and $t=1$ are fixed.
Similarly, a pressure field can be introduced as a Lagrange multiplier for the incompressibility constraint, whose properties are further detailed in Section \ref{Literature review} below.

Motivated by the classical entropic regularization of optimal transport and its connection with the Schr{\"o}dinger bridge problem \cite{leonard2013survey}, \cite{EntropicInt} introduced an analogous regularization of Brenier's incompressible optimal transport, which is known as the \emph{Br{\"o}dinger problem} -- in particular, this regularization restores uniqueness for the minimizer and can be seen as \emph{the Schrödinger bridge problem on the space of couplings}; we refer the reader to Section \ref{Statements geodesics} and Appendix \ref{Appendix} for a detailed presentation.

Exploiting the similarity between the constrained minimization problems \eqref{def distance} and  \eqref{eq:Bremutli}, we introduce an analog regularization for \eqref{def distance} in Section \ref{Statements geodesics}, which we study by adapting the methods developed in \cite{Baradatoptimality,SmallNoiseLimit} for the Br{\"o}dinger problem. As for the Br{\"o}dinger problem, we could have studied an entropic pathwise formulation, which we mention in the Appendix. We thus prove existence and uniqueness for Lagrange multipliers in both regularized and un-regularized settings, and we derive the related optimality conditions. The precise connection between the regularization and the original geodesic problem is given by a $\Gamma$-convergence result in Theorem \ref{Existence of geodesics}, as well as weak-$\star$ convergence for the multipliers in Proposition \ref{Proposition weak star convergence}.

\subsection{Literature review} \label{Literature review}

Hereafter, we review several contributions related to problems discussed in this article.

\paragraph{Transport maps on bounded domains.}
As previously mentioned, \cite{DacorognaMoser} showed a way to construct regular transport maps between two probability distributions that are supported on bounded sets. 
More precisely,
given a bounded set $\Omega \subset \R^d$ with a $C^{k+3,\alpha}$ boundary $\partial\Omega$, for some integer $k \geq 0$, $\alpha \in (0,1)$, and functions $f,g \in C^{k,\alpha}(\overline{\Omega})$ such that $f,g > 0$ on $\overline{\Omega}$ and $\int f \d x = \int g \d x$, \cite[Theorem 1]{DacorognaMoser} constructs a transport map $T$ from $f \d x$ to $g \d x$ that fixes the boundary $\partial\Omega$, and such that $T, T^{-1} \in C^{k+1,\alpha}(\overline{\Omega}, \R^d)$. 
This result was later improved by relaxing the H{\"o}lder-continuity of densities in \cite[Theorem 1, Theorem 2]{RiviereYe}, allowing for weaker conditions such as bounded or BMO densities.
\cite[Section 3]{RiviereYe} is based on a flow-map construction analogous to \cite{DacorognaMoser}, but using convolution as a regularizing interpolation instead of linear interpolation -- this method could work on $\R^d$ as well.
On the contrary, the proof of Theorem \ref{Theorem Dacorogna-Moser} below uses the linear interpolation of \cite{DacorognaMoser} without regularization.
Recalling Remark \ref{rem:addREg}, we emphasize that the bound on $T$ in Theorem \ref{Theorem Dacorogna-Moser} has one less degree of regularity than in the aforementioned works; however, the key feature of Theorem \ref{Theorem Dacorogna-Moser} is to prove \emph{global regularity estimates} on $\R^d$, contrary to these works.

In the recent literature, several strategies have been proposed to build transport maps on the whole $\R^d$, one popular method relying on the time-reversal of diffusion processes.

\paragraph{Time-reversed diffusion and Langevin dynamics.} \cite[Theorem 1.1]{KimMilmanHeatFlow} constructed a 1-Lipschitz transport map from $\mu ( \d x ) = e^{-U (x)} \d x$ to a measure $\nu = e^{-V (x) }\mu  ( \d x ) $, where $V$ is convex and $U \in C^{3,\alpha}_{loc}(\R^d)$, $\alpha > 0$, is a convex function with a specific shape that makes it sufficiently coercive. 
The construction relies on the generalized heat flow
\begin{equation*}
    \partial_t\mu_t = \Delta\mu_t - \nabla U \cdot \nabla\mu_t + \Delta U \mu_t, \quad \mu_0 = \nu,
\end{equation*} 
whose solution converges towards $\mu$ as $t \to +\infty$. 
Equivalently, $\mu_t$ is the law at time $t$ of the overdamped Langevin dynamics \begin{equation} \label{Langevin}
    \d X_t = -\nabla U(X_t)\d t + \sqrt{2}\d B_t, \quad X_0 \sim \nu,
\end{equation} 
driven by a Brownian motion $(B_t)_{t \geq 0}$.
Introducing the flow maps $S_t : \R^d \rightarrow \R^d$ satisfying
\begin{equation} \label{eq KIm Milman}
    \frac{\d}{\d t}S_t = - \nabla\log\frac{\d \mu_t}{\d \mu} \circ S_t, \quad S_0 = \text{Id},
\end{equation} 
we have $S_t \rightarrow S_\infty$ as $t \rightarrow + \infty$ with $\mu_t = (S_t)_{\#}\nu$, so that the desired Lipschitz transport map from $\mu$ to $\nu$ -- or \emph{heat flow map} -- is given by $T:=S_\infty^{-1}$.
We notice that the quantity $v_t:=-\nabla\log(\mu_t/\mu)$ in \eqref{eq KIm Milman} defines a vector field such that $(\mu_t,v_t)_{t \in [0,1]}$ satisfies the continuity equation. 
\cite[Theorem 1.1]{KimMilmanHeatFlow} then proves that $T$ is a contraction, similarly to Caffarelli's result for the optimal transport map \cite{CaffarelliContrOT}.
This construction was later leveraged to build Lipschitz transport maps under different assumptions.

The Langevin diffusion \eqref{Langevin} has been further used in \cite[Theorem 1]{TransportationLogLipschitz} to show existence for a bi-Lipschitz transport map from an $\alpha$-log concave measure $\mu = e^{-U}$, $\alpha > 0$, to any log-Lipschitz perturbation $\nu = e^{-U + V}$, i.e. $V$ is Lipschitz-continuous, under an additional boundedness assumption on $\nabla^3 U$. An extension to manifold settings was further proposed by \cite{pablo}.

Our extended Dacorogna-Moser construction also leverages the regularization and contraction properties of the Langevin diffusion \eqref{Langevin} to prove PDE  estimates in Proposition \ref{Lemma PDE estimates}. 
Regarding assumptions, we relax the strong log-concavity assumption on $\mu$ into the less restrictive requirement that $\mu$ is strictly asymptotically log-concave.
However, our regularity assumptions are slightly stronger, as we require both $\nabla U$ and $\nabla^2 U$ to be Lipschitz -- see Assumption \ref{Assumption 1} below.
We also need to assume that the log-perturbation $V$ is bounded and Lipschitz-continuous ($\varphi = e^{V}-1$ in our notation), whereas \cite[Theorem 1]{TransportationLogLipschitz} only requires $V$ being Lipschitz. 
This is related to the fact that we consider target measures of the shape $(1+\varphi)\cdot\mu$, and we control the Lipschitz constant directly in terms of $\varphi$, as it is more suitable for the geodesic problem studied in Section \ref{Statements geodesics}.
We furthermore quantify the higher-order regularity of the transport map in terms of the one of $\varphi$.

\paragraph{Gaussian setting and Ornstein-Uhlenbeck diffusion.} 
Various improvements have been proposed for the case where $\mu$ is the standard Gaussian, i.e. $U(x) = \frac{1}{2} \vert x \vert^2$.
\cite[Theorem 1, Theorem 2]{LipschitzPropAlongHeatFlow} proves the Lipschitz-continuity of the heat flow map when the target measure $\nu$ is strongly log-concave, or when $\nu$ is the convolution of the standard Gaussian with a compactly supported measure. 
To do so,
\cite[Lemma 3]{LipschitzPropAlongHeatFlow} establishes an explicit bound on the Lipschitz constant, which depends on the derivative of the velocity field in \eqref{eq KIm Milman} -- also called the \emph{score function} -- thus reducing the problem to bounding $-\nabla^2\log(\mu_t/\mu)$. 
This estimate is exploited in \cite[Theorem 1.4]{WeakConc} to prove existence for a Lipschitz map from the standard Gaussian to any strictly asymptotically log-concave measure $\nu$ such that $\int_0^1r \kappa^-_{-\log\nu}(r)\d r < +\infty$, recalling the definition \eqref{weak semiconvexity profile} of the weak semiconvexity profile. 

More generally, proving Lipschitz bounds on the score function for the Ornstein-Uhlenbeck process has become a common strategy, which is used in \cite[Theorem 1]{BiLipschitzScore}  
to prove existence of bi-Lipschitz transport maps from the standard Gaussian to a general class of measures, which includes Gaussian mixtures and convolutions of Gaussians with compactly supported measures. 
Similarly, by studying the regularity of the score function, \cite[Theorem 2]{ArtStepHolderBootstrap} proves that the heat flow map belongs to $C^{\lfloor\beta\rfloor + 1}$ when the target measure $\nu$ is a bounded $\beta$-H{\"o}lder perturbation of the standard Gaussian, allowing for quantifying the regularity of the map depending on the one of the perturbation. 

Our results in Theorem \ref{Theorem Dacorogna-Moser} are less precise than the aforementioned ones on Hölder-perturbations of Gaussians, but our setting is more general as it allows for any initial distribution $\mu$ that is strictly asymptotically log-concave -- not necessarily a Gaussian. 
For the sake of completeness, we finally mention the works \cite{FirstStabKimMilm,StabilityNearLipsc}, which prove quantitative stability for the heat-flow map with respect to the target distribution $\nu$, in the Gaussian case. 

\paragraph{Other interpolations.} Many other methods to construct transport maps have been proposed in the recent literature. 
A famous strategy introduced by \cite{FlowMatching} is \emph{flow matching}, which amounts to using the flows of probability measures and vector fields given by \begin{equation*}
\mu_t(x) = \int \mu_t(x|x_1)\d \nu(x_1), \qquad v_t(x) = \int v_t(x|x_1) \frac{\mu_t(x|x_1)}{\mu_t(x)}\d \nu(x_1),
\end{equation*} 
where $\nu$ is the target distribution, $\mu_t(x|x_1)$ is a conditional distribution such that $\mu_0(x|x_1) = \mu(x)$ -- so that $\mu_0 = \mu$ -- and $\mu_1(x|x_1)$ is concentrated around $x_1$ -- so that $\mu_1$ approximates $\nu$. 
A possible choice of conditional measure is the Gaussian
\begin{equation*}
\mu_t(x|x_1) = \mathcal{N}(x|m_t(x_1), \sigma_t^2(x_1)I_d),
\end{equation*} 
for some suitable choice of the conditional mean $m_t$ and variance $\sigma^2_t$.
If $\sigma_t$ only depends on time for simplicity, $\sigma_0 = 1 = 1-\sigma_1$ and $m_t(x_1) = t x_1$, then $\mu_t$ corresponds to the law at time $t$ of the stochastic process 
\begin{equation*}
   X_t = m_t(Y) + \sigma_t \xi, \quad Y \sim \nu, \quad \xi \sim \mathcal{N}(0, I_d), 
\end{equation*}
so that $(X_t)_{t \in [0,1]}$ is a stochastic interpolation between the initial Gaussian distribution and the target measure.
This construction can be extended to other stochastic interpolants between arbitrary distributions, as in \cite{StochInterp} that studies the laws of the process \begin{equation*}
    X_t = m_t(X_0, Y), \quad Y \sim \nu, \quad X_0 \sim \mu,
\end{equation*} 
and further generalized in \cite{VariantStochInterp} as
\begin{equation*}
    X_t = m_t(X_0,Y) + \sigma_t\xi, \quad \xi \sim \mathcal{N}(0,I_d).
\end{equation*} 
For different choices of $m_t$ and $\sigma_t$, \cite{ArtStephanovitchWeakConcFlow} proves dimension-free bounds for the Lipschitz constant of the resulting transport map from the standard Gaussian to a target measure of the form $ e^{- V(x) + W(x)} \d x$, where $V \in C^2(\R^d, \R)$ is $\alpha$-convex, $\alpha > 0$, and $W$ is $\beta$-Hölder, $\beta \in (0,1]$. 
In particular, \cite[Corollary 1]{ArtStephanovitchWeakConcFlow} covers the Lipman flow matching setting, and \cite[Corollary 2]{ArtStephanovitchWeakConcFlow} considers the stochastic-interpolant flow matching setting. 
Furthermore, the aforementioned Ornstein-Uhlenbeck setting can be seen as a specific instance of this type of interpolation for $m_t(x,y) = t y$ and $\sigma_t = \sqrt{1 - t^2}$, and the Lipschitz-continuity of the resulting transport map is proved in \cite[Corollary 3]{ArtStephanovitchWeakConcFlow}. 
The target measure is strictly asymptotically log-concave when $\beta = 1$, and this result overlaps \cite[Theorem 1.4]{WeakConc} in  this case. 
For a general value $\beta \in (0,1)$, neither of these two results implies the other one.

Once again, we emphasize that this setting is different from ours, as the starting measure $\mu$ in Theorem \ref{Theorem Dacorogna-Moser} needs not be a Gaussian, and we study the regularity of the Dacorogna-Moser transport map with respect to the one of the perturbation $\varphi$ in the target measure $(1+\varphi)\cdot\mu$.

\paragraph{Geodesics in the space of couplings.}
The basic references \cite{ConfLackPal,LouisLacker} on the geodesic problem in $\Pi(\mu,\nu)$ have already been discussed in Section~\ref{Subsection intro geodesics}. We therefore focus here on related variational problems that motivate our analysis.

Our study of Lagrange multipliers is inspired by Brenier's multiphase formulation of incompressible Euler equations \cite{Brenier1999MinimalGO}. In that setting, the pressure arises as the Lagrange multiplier associated with the incompressibility constraint and is shown to be unique as a distribution, with $\nabla p \in \mathcal{M}((0,1)\times\mathbb{T}^d)$. The regularity theory was subsequently refined in \cite{RegPressure}, which proves that $\nabla p \in L^2((0,1),\mathcal{M}(\mathbb{T}^d))$ and $p \in L^2((0,1),L^{d/(d-1)}(\mathbb{T}^d))$. Determining the optimal regularity of the pressure remains a challenging open question, see for instance the semiconcavity conjecture in \cite{Brenier2013OptimalReg}. In contrast, our work establishes existence, uniqueness, and rigorous optimality conditions for the Lagrange multipliers associated with the marginal constraints, but does not address their optimal regularity. Whether or not the techniques of \cite{RegPressure} can be adapted to our setting is an interesting open question.

The regularized problem considered in this article is similarly motivated by the entropic regularization of Brenier's formulation. 
This latter regularization was introduced in \cite{EntropicInt} as the \emph{Br\"{o}dinger problem}; it amounts to adding a Fisher-information term to the kinetic energy -- see Appendix \ref{Appendix}. In particular, this regularization restores uniqueness of the minimizer, which does not hold in general for the unregularized Brenier problem. 
Building on this formulation, \cite{Baradatoptimality} establishes existence, uniqueness, and optimality conditions for the pressure field, while \cite{SmallNoiseLimit} proves $\Gamma$-convergence towards Brenier's multiphase problem together with weak-$\star$ convergence of the associated multipliers.

Our approach follows the same general strategy, but substantial modifications are required. In particular, the heat-flow regularization used in \cite{SmallNoiseLimit} does not preserve the constraint set $\Pi(\mu,\nu)$ and therefore cannot be used to construct recovery sequences. Instead, assuming that $\nabla^2\log\mu$ and $\nabla^2\log\nu$ are bounded, we regularize by the gradient flow of the relative entropy $H(\cdot\mid\mu\otimes\nu)$ in $(\mathcal P_2(\mathbb R^d\times\mathbb R^d),\mathcal W_2)$, which leaves $\Pi(\mu,\nu)$ invariant. The required regularization estimates are obtained using the theory of the Schr\"{o}dinger bridge problem on general metric spaces developed in \cite{GenMetricSpace}.

\subsection{Notations} \label{Notations}
We introduce some notations that we will use throughout this article.
\begin{itemize}
    \item We define $\nabla_\mu \cdot \xi := \frac{\nabla \cdot (\mu \xi)}{\mu} = \nabla\log\mu \cdot \xi + \nabla\cdot\xi$, for every $\xi$ such that its derivative is well-defined. In particular, for every $\psi$ such that $\nabla\psi$ and $\nabla^2\psi$ are well-defined, we set $\Delta_\mu \psi := \nabla_\mu \cdot \nabla\psi$.

    This notation is motivated by the fact that $\nabla_\mu \cdot \xi$ is the adjoint of the gradient in $L^2 ( \mu )$.
    \item $H(\cdot|\mu \otimes \nu)$ denotes the relative entropy with respect to $\mu \otimes \nu$, i.e. \begin{equation*}
        H(\rho |\mu\otimes\nu):= \begin{cases} \displaystyle
            \int_{\R^d \times \R^d}\log\frac{\d \rho}{\d \mu \otimes \nu}\d \rho  \quad \text{if} \ \rho \ll \mu\otimes \nu \\
            +\infty \quad \text{otherwise}.
        \end{cases}
    \end{equation*}
    \item $\lambda^-:= \max\{-\lambda,0\}$ denotes the negative part of $\lambda \in \R$.
    \item $\lesssim$ means ``less than or equal to up to a numerical constant''.
    \item For a measure $\mu$ on $\R^d$ and a map $T: \R^d \to \R^d$, we denote by $T_{\#}\mu$ the pushforward of $\mu$ through the transport map $T$, i.e. \begin{equation*}
        (T_{\#}\mu)(A) = \mu(T^{-1}(A)), \quad \forall A \subseteq \R^d \ \text{measurable}.
    \end{equation*}
\end{itemize}

\section{Statement of the main results} \label{Section main results}

In this section, we state our main results on the extended Dacorogna-Moser construction and the geodesic problem.

\subsection{The extended Dacorogna-Moser construction}

Our main result Theorem \ref{Theorem Dacorogna-Moser} requires the following assumption.

\begin{assumption} {~} \label{Assumption 1} 
\begin{enumerate}[label = (\roman*)]
\item $\mu$ is strictly asymptotically log-concave -- see Definition \ref{def asymptotic convexity};
\item $\nabla^j\log\mu$ is well-defined and Lipschitz-continuous for every $j \in \{1, \ldots, n+1\}$ for some $n \geq 0$.
\end{enumerate}

\end{assumption}

As mentioned in Section \ref{sec:intro}, the transport map in Theorem \ref{Theorem Dacorogna-Moser} is constructed from a flow of probability measures $(\rho_s)_{s\in [0,1]}$ between $\mu$ and $(1+\varphi)\cdot\mu$, for a suitable choice of velocity vector field $(v_s)_{s \in [0,1]}$ such that the continuity equation \eqref{eq:transportDM} is satisfied. 
The transport map is then given by the flow generated by the velocity at the terminal time. 
The simplifying assumptions $\rho_s = (1+s\varphi)\cdot\mu$ and $v_s =- \frac{\nabla\psi}{1+s\varphi}$ reduce the continuity equation to solving the equation $\Delta_\mu\psi = \varphi$ -- recalling the $\Delta_\mu$ notation introduced in Section \ref{Notations}. 
Therefore, the desired estimates on the transport map in Theorem \ref{Theorem Dacorogna-Moser} result from estimating the solution of this equation. 
This is the content of the following result, whose proof can be found in Section \ref{Proofs gen DM}. 
For convenience in writing the proof, the result is stated for the operator $\frac{1}{2}\Delta_\mu$ rather than $\Delta_\mu$, but the resulting estimates are equivalent.

\begin{prop}[Parabolic and elliptic PDE estimates]\label{Lemma PDE estimates}
Let $\mu$ be a probability measure on $\R^d$ that satisfies Assumption \ref{Assumption 1} for some $n \geq 0$, and let $\varphi \in W^{n,\infty}(\R^d,\R)$ be such that $\int_{\R^d}\varphi\d\mu = 0$. 
Let $(S_t)_{t \geq 0}$ be the semigroup generated by the operator $\frac{1}{2}\Delta_\mu$ defined in Section \ref{Notations}. Then: \begin{enumerate}[label = (\roman*)]
    \item $\nabla S_t\varphi \in W^{n,\infty}$ for every $t > 0$, and there exist constants $c, \alpha >0$ that only depend on $\mu$ and $n$ such that \begin{equation} 
    \begin{cases}\label{estimates parabolic PDE}
        \lVert \nabla S_t \varphi\rVert_{W^{n,\infty}} \leq c e^{-\alpha t}\lVert \varphi\rVert_{W^{n,\infty}}, \quad \forall t \geq 1, \\
        \lVert \nabla S_t\varphi\rVert_{W^{n,\infty}} \leq \frac{c}{\sqrt{t}}\lVert \varphi\rVert_{W^{n,\infty}}, \quad \forall t \in (0,1].
        \end{cases}
    \end{equation}
    \item The system \begin{equation} \label{system elliptic}
            \frac{1}{2}\Delta_\mu\psi = \varphi , \qquad \int_{\R^d}\psi \d\mu = 0,
    \end{equation} 
    admits a unique solution $\psi \in H^1(\mu)$, which is given by $\psi = - \int_0^{+\infty}S_t\varphi\d t$. Furthermore, $\nabla\psi \in W^{n,\infty}$ and \begin{equation} \label{estimates elliptic PDE}
        \lVert \nabla\psi\rVert_{W^{n,\infty}} \leq C\lVert \varphi\rVert_{W^{n,\infty}},
    \end{equation} where $C > 0$ is a constant that only depends on $\mu$ and $n$.
    \end{enumerate} 
\end{prop}

\begin{rmk}[Explicit constants]
The constants that appear in the estimates \eqref{estimates parabolic PDE}-\eqref{estimates elliptic PDE} can be made explicit, even though their expression is quite heavy -- see \cite[Theorem 1]{EberleZimmer} and \cite[Theorem 3.4]{PriolaWangEst}. 
These expressions depend on the weak log-concavity profile $\kappa_{-\log\mu}$ of the reference measure.
\end{rmk}

The proof of \eqref{estimates parabolic PDE} relies on the Feynman-Kac stochastic representation formula and coupling by reflection, leveraging the estimates proved in \cite{PriolaWangEst} and \cite{EberleZimmer}.
The bound \eqref{estimates elliptic PDE} follows, as we can verify that the solution is formally given by $\psi = -\int_0^\infty S_t\varphi\d t$ and then leverage the estimates on $S_t\varphi$.

\begin{rmk}[Regularity of the profile] \label{rem:profile}
To apply the estimates \cite[Theorem 1]{EberleZimmer} and \cite[Theorem 3.4]{PriolaWangEst}, we also need to assume \begin{equation*}
        \int_0^1r\kappa_{-\log\mu}^-(r)\d r < +\infty,
    \end{equation*} but this is implied by (ii) in \ref{Assumption 1}, because $\nabla\log\mu$ is Lipschitz-continuous.
\end{rmk}

\subsection{The geodesic problem and its regularization} \label{Statements geodesics}

In this section, we state our main results about the geodesic problem, the proofs being deferred to Section \ref{proofs geodesic}. First, we introduce a regularization of the geodesic problem using the relative Fisher information, taking inspiration from the regularization of Brenier's multiphase formulation of incompressible fluids given by the Br{\"o}dinger's problem \cite{EntropicInt}. 

\begin{defn}[The regularized problem] \label{Def reg functional and reg prob}
Let $\rho_0, \rho_1 \in \Pi(\mu,\nu)$. For every $\varepsilon \in [0,1]$, we define \begin{equation*} 
    \mathcal{A}_{\varepsilon}(\rho):= \inf\bigg\{\frac{1}{2}\int_0^1\int_{\R^d \times \R^d}|v_t|^2\d\rho_t\d t + \frac{\varepsilon^2}{8}\int_0^1\int_{\R^d \times \R^d}\bigg\lvert \nabla\log\frac{\d \rho_t}{\d \mu\otimes\nu}\bigg\rvert^2\d\rho_t \d t, \ \partial_t\rho + \nabla\cdot(\rho v) = 0\bigg\},
\end{equation*} and we consider the minimization problem \begin{equation} \label{def regularization}
    \inf\big\{\mathcal{A}_{\varepsilon}(\rho), (\rho_t)_{t \in [0,1]} \subset \Pi(\mu,\nu), (\rho\vert_{t=0}, \rho\vert_{t=1}) = (\rho_0,\rho_1)  \big\}.
\end{equation}
\end{defn}

We refer to Appendix \ref{Appendix} for an alternative formulation motivating this regularized problem, which can be seen as \emph{the Schrödinger bridge problem on the space of couplings}.

\begin{rmk}[Geodesic problem]
When $\varepsilon = 0$, the minimization problem \eqref{def regularization} is the geodesic problem \eqref{def distance}.
\end{rmk}

For both the geodesic problem and its regularization \eqref{def regularization}, we prove finiteness and existence of minimizers, and we show convergence as $\varepsilon$ tends to zero. 
In particular, the following result establishes existence of geodesics in the strict asymptotic log-concavity setting and proves $\Gamma$-convergence of the regularized problem \eqref{def regularization} towards the geodesic problem. 

\begin{thm}[Existence of geodesics and $\Gamma$-convergence] \label{Existence of geodesics}
Suppose that $\mu$ and $\nu$ satisfy Assumption \ref{Assumption 1} for $n = 0$, and let $\rho_0, \rho_1 \in \Pi(\mu,\nu)$. 
Then, the problem \eqref{def regularization} with $\varepsilon = 0$ admits at least one minimizer. Furthermore, assuming that $H(\rho_0|\mu\otimes\nu), H(\rho_1|\mu\otimes\nu) < +\infty$, then for every $\varepsilon \in (0,1]$ the minimization problem \eqref{def regularization} is finite, it admits a unique minimizer and \begin{equation*}
    \Gamma - \lim_{\varepsilon \to 0}\mathcal{A}_{\varepsilon} + i_{01} = \mathcal{A}_{0} + i_{01},
\end{equation*} for the uniform convergence in $C([0,1],\Pi(\mu,\nu))$ and for the pointwise-in-time convergence with respect to the topology of weak convergence of measures on $\Pi(\mu,\nu)$, where $i_{01}(\gamma) = 0$ if $(\gamma_0,\gamma_1) = (\rho_0,\rho_1)$ and $i_{01}(\gamma) = +\infty$ otherwise.
\end{thm}

Existence of minimizers is ensured by \cite{LouisLacker} in the strong log-concavity setting, and we extend this result to the general case using a Lipschitz-continuous transport map from the standard Gaussian to any strictly asymptotically log-concave measure \cite[Theorem 1.4]{WeakConc}. 
The $\Gamma$-convergence proof leverages estimates contained in \cite{GenMetricSpace} for the Schr{\"o}dinger problem on general metric spaces.

\begin{rmk}[Sufficient assumption]
The only properties needed for the proof of Theorem \ref{Existence of geodesics} are the existence of Lipschitz-continuous transport maps from the standard Gaussian distribution to the marginals $\mu$ and $\nu$, and the fact that $\nabla^2\log\mu$ and $\nabla^2\log\nu$ are bounded.
\end{rmk}

To better understand the structure of geodesics and minimizers of the regularized problem, we are interested in deriving optimality conditions. 
To do so, we show existence of a Lagrange multiplier associated to the marginal constraints, and we then derive rigorous optimality conditions. 
This first requires defining the appropriate notion of perturbation.

\begin{defn}[Space of perturbations and modified problem] \label{Perturbations and modified problem}
We define the space $S_\mu$ of functions $\varphi : [0,1] \times \R^d \rightarrow \R$ that satisfy, writing $\varphi = (\varphi_t)_{t \in [0,1]}$, \begin{align*}
& \varphi_0 \equiv \varphi_1 \equiv 0 , \qquad \int_{\R^d}\varphi\d\mu = 0, \\
    &\varphi \in L^\infty([0,1], W^{2,\infty}(\R^d)) \cap W^{1,2}((0,1),L^{\infty}(\R^d)).
\end{align*} We similarly define the space $S_\nu$ and we equip both spaces with the norm \begin{equation*}
    N(\varphi):= \esssup_{t \in [0,1]}\lVert \varphi_t\rVert_{W^{2,\infty}(\R^d)} + \bigg(\int_0^1\lVert\partial_t\varphi_t\rVert^2_{L^{\infty}(\R^d)}\d t\bigg)^{1/2}.
\end{equation*}
Now, fix two endpoints $\rho_0,\rho_1 \in \Pi(\mu,\nu)$ such that $H(\rho_0|\mu\otimes\nu), H(\rho_1|\mu\otimes\nu) < +\infty$. For every $(\mu_t)_{t \in [0,1]}, (\nu_t)_{t \in [0,1]} \in C([0,1],\mathcal{P}_2(\R^d))$ such that $\mu_0=\mu_1 = \mu$ and $\nu_0=\nu_1 = \nu$, we define the functional 
\begin{equation} \label{definition modified problem}
    \mathcal{F}_{\varepsilon}(\mu_\cdot,\nu_\cdot) := \inf\big\{\mathcal{A}_{\varepsilon}(\rho), \rho_t \in \Pi(\mu_t,\nu_t) \ \forall t \in [0,1], (\rho\vert_{t=0},\rho\vert_{t=1}) = (\rho_0,\rho_1) \big\}, \quad \forall \varepsilon \in [0,1].
\end{equation} If $\mu_t = (1+\varphi_t)\cdot\mu$, $\nu_t = (1+\psi_t)\cdot\nu$ for $\varphi \in S_\mu, \psi \in S_\nu$, then we denote \eqref{definition modified problem} by $\mathcal{F}_{\varepsilon}(\varphi,\psi)$; if either $1+\varphi$ or $1+\psi$ is not non-negative, then we set $\mathcal{F}_{\varepsilon}(\varphi,\psi):=+\infty$.
\end{defn}

Using this notion of perturbation for the marginal constraints, we now establish existence of the Lagrange multiplier for both the geodesic problem and its regularization. 

\begin{thm}[Existence of the Lagrange multiplier] \label{Existence of the Lagrange multiplier} Suppose that $\mu$ and $\nu$ satisfy Assumption \ref{Assumption 1} with $n = 2$. Then for every $\varepsilon \in [0,1]$, there exists $p^{\varepsilon}$ in the convex sub-differential $\partial\mathcal{F}_{\varepsilon}(0,0) \subset ( S_\mu \times S_\nu )'$ of $\mathcal{F}_\varepsilon$ at $(0,0)$, meaning that 
\begin{equation} \label{dual inequality}
    \mathcal{F}_{\varepsilon}(\varphi,\psi) \geq \mathcal{F}_{\varepsilon}(0,0) + \langle p^{\varepsilon}, (\varphi,\psi)\rangle \quad \forall \varphi \in S_\mu,\psi \in S_\nu.
\end{equation}
    
\end{thm}

The proof shows that $\mathcal{F}_{\varepsilon}$ is bounded near the origin in $S_\mu \times S_\nu$, and then deduces existence of an element in $\partial\mathcal{F}_{\varepsilon}(0,0)$ by convexity and lower semicontinuity of $\mathcal{F}_{\varepsilon}$. The boundedness of $\mathcal{F}_{\varepsilon}(\varphi,\psi)$ near the origin crucially relies on the extended Dacorogna-Moser construction from Theorem \ref{Theorem Dacorogna-Moser} to build transport maps from $\mu$ and $\nu$ to the perturbed marginals, before using them as competitors for the modified problem \eqref{definition modified problem}. 
We then leverage the estimates on these maps to bound $\mathcal{F}_{\varepsilon}(\varphi,\psi)$.

Now knowing existence for Lagrange multipliers, we can perform a first-order expansion of the inequality \eqref{dual inequality} to derive optimality conditions. 
However, this expansion requires more regularity for the perturbations, motivating the following definition.

\begin{defn}[Regular spaces of perturbations]
We say that $\varphi$ belongs to the space $\tilde{S}_\mu$ if $\varphi \in S_\mu$ and $\varphi \in L^\infty([0,1],W^{3,\infty}(\R^d))\cap W^{1,2}([0,1],W^{1,\infty}(\R^d))$, and we equip it with the same norm $N(\cdot)$ as the one on $S_\mu$. We similarly define the space $\tilde{S}_\nu$.
\end{defn}

The following result now gives a rigorous meaning to the formal optimality conditions \eqref{formal optimality conditions}, and deduces uniqueness for the Lagrange multipliers of both the geodesic problem and its regularization on the space $(\tilde{S}_\mu \times \tilde{S}_\nu)'$.

\begin{thm}[Optimality conditions and uniqueness for the multipliers] \label{Optimality condition for the Lagrange multiplier}
Suppose that $\mu$ and $ \nu $ satisfy Assumption \ref{Assumption 1} with $n = 3$. For $\varepsilon \in [0,1]$, let $(\rho,v)$ be a minimizer for $\mathcal{F}_{\varepsilon}(0,0)$ and $p^{\varepsilon} \in \partial\mathcal{F}_{\varepsilon}(0,0)$. Then, for every $\varphi \in \tilde{S}_\mu$ and $ \psi \in \tilde{S}_\nu$, we have \begin{equation} \label{optimality conditions}
    \langle p^{\varepsilon}, (\varphi,\psi) \rangle  = \langle \partial_t(\rho v) + \operatorname{div}\{(v \otimes v - w^\varepsilon \otimes w^\varepsilon )\rho\}, (\nabla f,\nabla g)\rangle
\end{equation}  where \begin{equation*}
 w^{\varepsilon}:= \frac{\varepsilon}{2}\nabla\log\frac{\d \rho}{\d \mu\otimes\nu}, \quad \forall \varepsilon \in (0,1],
\end{equation*} $w^0:= 0$ and $f$, $g$ are respectively the solution of $\Delta_\mu f = \varphi$ with $\int_{\R^d}f \d\mu = 0$, and $\Delta_\nu g = \psi$ with $\int_{\R^d}g \d\nu = 0$. 
In particular, since \eqref{optimality conditions} holds for every minimizer $(\rho,v)$, this implies uniqueness for $p^{\varepsilon}$ in $(\tilde{S}_\mu \times \tilde{S}_\nu)'$, for every $\varepsilon \in [0,1]$.
    
\end{thm}

Our derivation of optimality conditions is strongly inspired by \cite{Baradatoptimality}. 
In this work, the quantity $w^\varepsilon$ is substituted by $\frac{\varepsilon}{2}\nabla\log\rho$, and the marginal constraint is replaced by the incompressibility constraint, which corresponds to imposing the Lebesgue measure as one of the marginals. 
Moreover, \cite[Lemma 14]{Baradatoptimality} proves optimality conditions for the distributional gradient of the pressure against $C^{1,2}$ test functions, whereas the adequate space of test functions in our case is $\tilde{S}_\mu \times \tilde{S}_\nu$, requiring test functions to have three derivatives with respect to space. 
This restriction is due to the fact that $W^{3,\infty} (\R^d)$ is not dense in $W^{2,\infty} \cap C^2 (\R^d)$ -- whereas $W^{3,\infty} ( \mathbb{T}^d )$ is dense in $C^{2}  ( \mathbb{T}^d )$ --, preventing us from extending \eqref{optimality conditions} by density.
For the same reason, our weak-$\star$ convergence result Proposition \ref{Proposition weak star convergence} for multipliers only holds in $(\tilde{S}_\mu \times \tilde{S}_\nu)'$, whereas the analog result for the Br{\"o}dinger problem \cite[Theorem 2.6]{SmallNoiseLimit} considers test functions in $C^{1,2}([0,1]\times\mathbb{T}^d)$. We state our weak-$\star$ convergence result hereafter.

\begin{prop}[Convergence of the Lagrange multipliers] \label{Proposition weak star convergence}
For every $\varepsilon \in [0,1]$, let
$p^{\varepsilon}$ be the unique element of  $\partial \mathcal{F}_{\varepsilon}(0,0)$ in $(\tilde{S}_\mu \times \tilde{S}_\nu)'$. Then, for every $\varphi \in \tilde{S}_\mu$ and $ \psi \in \tilde{S}_\nu$, we have \begin{equation}
    \langle p^{\varepsilon}, (\varphi,\psi) \rangle \rightarrow \langle p^0, (\varphi,\psi) \rangle.
\end{equation}

\section{Proofs}

\normalfont{This section is devoted to the proof of the results stated in Section \ref{Section main results}.}

\subsection{The generalized Dacorogna-Moser construction}  \label{Proofs gen DM}
 \begin{proof}[Proof of the parabolic and elliptic PDE estimates (Proposition \ref{Lemma PDE estimates})]
We first prove the estimates on the parabolic PDE, before leveraging them to prove the estimates on the elliptic one.

\paragraph{Parabolic PDE.}
For convenience, we consider the time-reversed PDE with terminal time $T > 0$: 
\begin{equation} \label{time inverted PDE}
 \partial_t v_t + \frac{1}{2}\Delta_\mu v_t = 0, \quad v_T = \varphi.
\end{equation} We claim that there exist positive constants $\alpha, c, (c_{\delta})_{\delta >0}$ such that, for every $0\leq j \leq n$, \begin{align}
    \lVert \nabla^{j+1}v_t\rVert_{\infty} & \leq \label{exp decay} c_{\delta}e^{-\alpha(T-t)}\lVert \varphi\rVert_{W^{j,\infty}}, \qquad \forall t < T - \delta, \\
     \lVert \nabla^{j+1}v_t\rVert_{\infty} & \leq \label{square root decay} \frac{c}{\sqrt{T-t}}\lVert \varphi\rVert_{W^{j,\infty}}, \qquad \forall t \in [T-\delta \vee 0, T) .
\end{align} 

This would imply the desired estimates \eqref{estimates parabolic PDE}. Indeed, $(v_t)_{t \in [0,T]}:=(S_{T-t}\varphi)_{t \in [0,T]}$ satisfies \eqref{time inverted PDE}, so that for $T \geq 1$ the norm $\lVert \nabla^{j+1}v_0\rVert_{\infty}$ decays exponentially and for $T \in (0,1)$ it behaves like the inverse of $\sqrt{T}$.

Now we prove our claim by induction. First, since $b:=\frac{1}{2}\nabla\log\mu$ is Lipschitz-continuous and $\varphi$ is bounded, \cite[Chapter 3, Theorem 13]{FriedmanPDE} guarantees that \eqref{time inverted PDE} has a unique $C^{1,2}$ solution $v_t$, whose derivatives are H{\"o}lder-continuous, and it satisfies $\lVert v_t \rVert_\infty \leq \lVert \varphi \rVert_\infty$. 
Introducing the solution of the stochastic differential equation
\begin{equation} \label{eq:Diffx}
\d X_s^{t,x} = b(X_s^{t,x})\d s + \d B_s, \quad s \geq t, \quad X_t^{t,x} = x, 
\end{equation} 
for $x\in \R^d$ and a Brownian motion $(B_s)_{s \geq t}$ on a probability space $(\Omega,\mathcal{F},\mathbb{P})$, the Feynman-Kac's representation formula \cite[Theorem 5.3]{FriedmanSDE} gives that
\begin{equation} \label{Feynman Kac}
    v_t(x) = \mathbb{E}[\varphi(X_T^{t,x})].
\end{equation} 
For $y \in \R^d$, the coupling by reflection procedure \cite[Theorem 1]{EberleZimmer} constructs a Brownian motion $( \hat{B}_s )_{s \geq t}$ on the same probability space -- that depends on $x$ and $(B_s)_{s \geq t}$ -- such that the corresponding solution $(\hat{X}^{t,y}_s)_{s \geq t}$ of \eqref{eq:Diffx} starting from $y$ 
satisfies 
\begin{equation} \label{est prob}
    p_s(x,y):= \mathbb{P}(X_s^{t,x}  \neq \hat{X}_s^{t,y})  \leq \frac{a}{e^{\beta (s-t)}-1}|x-y|,
\end{equation} where $\beta > 0$ is a constant that depends on $\mu$. Therefore, for some constants $(k_\delta)_{\delta > 0}$ we have $p_s(x,y) \leq k_\delta e^{-\beta(s-t)}|x-y|$ for every $s \geq t+\delta$. Now, using \eqref{Feynman Kac} and the boundedness of $\varphi$ we get \begin{equation*}
    |v_t(x)-v_t(y)| \leq \mathbb{E}[|\varphi(X_T^{t,x})-\varphi(\hat{X}_T^{t,y})|] \leq 2\lVert \varphi\rVert_\infty p_T(x,y),
\end{equation*} 
so that the estimate on $p_s(x,y)$ concludes the proof for $t \in [0, T-\delta)$. 
For $t \in [T-\delta, T)$, the gradient estimate \cite[Theorem 3.4]{PriolaWangEst} reads
\begin{equation*}
    \lVert\nabla v_t \rVert_\infty \leq \frac{2a}{\sqrt{(T-t)\wedge 1}}\lVert \varphi\rVert_\infty
\end{equation*} 
where $a > 0$ only depends on $\mu$, recalling Remark \ref{rem:profile} for verifying the required assumptions. 
This concludes the proof of the base case of our induction.

Let us now assume that \eqref{exp decay} and \eqref{square root decay} hold with common constants $c_\delta, c, \alpha$ for every $0\leq j \leq m-1\leq n-1$ and that the derivatives are H{\"o}lder-continuous. Up to changing the constants $(c_\delta)_{\delta > 0}$ and $(k_\delta)_{\delta > 0}$, we can assume that $\beta = \alpha$. 
Let $\partial_I$ denote the partial derivative with respect to the indices $I=\{x_1,\ldots,x_m\}$, so that formally differentiating \eqref{time inverted PDE} with respect to $\partial_I$ gives 
\begin{equation} \label{new PDE induction}
    \partial_t\partial_Iv_t + \frac{1}{2}\Delta_\mu \partial_Iv_t = - \sum_{J \subseteq I, |J|\geq 1}\partial_Jb \cdot \nabla \partial_{I\setminus J}v_t , \quad \partial_I v_T = \partial_I \varphi.
\end{equation} Since $\partial_Jb$ is Lipschitz-continuous -- because $\partial_J$ has at most $m$ derivatives -- and $\nabla\partial_{I\setminus J}v_t$ is H{\"o}lder-continuous for $t < T$  by the induction hypothesis -- $\partial_{I\setminus J}$ has at most $m-1$ derivatives --, the right-hand side in \eqref{new PDE induction} is continuous. 
Since $b$ is Lipschitz-continuous and the initial data $\partial_I \varphi$ is bounded, by \cite[Chapter 3, Theorem 13]{FriedmanPDE} we deduce that this equation is satisfied in the classical sense by $\partial_I v$, that $\partial_I v_t \in C^{1,2}$ for $t < T$, that the derivatives are H{\"o}lder-continuous and that $\partial_I v_t$ is continuous on $[0,T]$. Therefore, we can apply It\^o's formula to $(\partial_I v_s)_{s \in [0,T]}$ and the process $(X_s^{t,x})_{s \geq t}$, and take expectations to get \begin{equation} \label{equation induction step}
    \partial_I v_t(x) = \mathbb{E}[\partial_I \varphi(X_T^{t,x})] + \sum_{J \subseteq I, |J|\geq 1}\int_t^T \mathbb{E}[\partial_J b \cdot \nabla\partial_{I\setminus J}v_s \circ X_s^{t,x}]\d s.
\end{equation} 
Now, we proceed as for the base case. We fix $x,y \in \R^d$ and, for $t \in [T-\delta, T)$, we apply the gradient estimate \cite[Theorem 3.4]{PriolaWangEst} and get \begin{equation*}
    |\partial_I v_t(x) - \partial_I v_t(y)| \leq 2\lVert \partial_I \varphi\rVert_\infty \frac{a}{\sqrt{(T-t)\wedge 1}}|x-y| + 2C_b|x-y| \int_t^T\lVert \nabla v_s\rVert_{W^{m-1,\infty}}\frac{a}{\sqrt{(s-t)\wedge 1}}\d s,
\end{equation*} where $C_b := m \lVert \nabla b\rVert_{W^{m-1,\infty}}$ and we bounded every $\lVert \nabla\partial_{I \setminus J}v_s\rVert_\infty$ by $\lVert \nabla v_s\rVert_{W^{m-1,\infty}}$. By the estimates on $\lVert \nabla v_s\rVert_{W^{m-1,\infty}}$ given by the induction hypothesis we have \begin{equation*}
    |\partial_I v_t(x) - \partial_I v_t(y)| \leq \frac{2a}{\sqrt{T-t}}\lVert \nabla^m \varphi\rVert_\infty |x-y| + 2mC_b|x-y|\int_t^T\frac{ac}{\sqrt{s-t}\sqrt{T-s}}\d s \lVert \varphi\rVert_{W^{m-1,\infty}}
\end{equation*} which implies the bound on $\lVert \nabla^{m+1}v_t\rVert_\infty$ for $t \in [(T-1)\vee 0, T)$ since the second term is bounded in time on $[0,T]$. 

Now suppose that $0\leq t \leq T-\delta $ for $\delta > 0$ small enough.
Using \eqref{equation induction step} to compute the difference $|\partial_I v_t(x) - \partial_I v_t(y)|$, we can use coupling by reflection \cite[Theorem 1]{EberleZimmer} and \eqref{est prob} as previously to bound the first term resulting from \eqref{equation induction step}:
\begin{equation} \label{est first term exp case}
\vert \mathbb{E}[\partial_I \varphi(X_T^{t,x})] - \mathbb{E}[\partial_I \varphi(\hat{X}_T^{t,y})] \vert \leq 2\lVert \partial_I\varphi\rVert_\infty p_T(x,y) \leq 2\lVert \partial_I\varphi\rVert_\infty k_{\delta} e^{-\alpha(T-t)}|x-y|.
\end{equation} 
For the second term resulting from \eqref{equation induction step}, we split the integral in three ones on the intervals $[t, t+\delta/2]$, $[t+\delta/2, T-\delta/2]$ and $[T-\delta/2, T]$. 
Combining the gradient estimate of \cite[Theorem 3.4]{PriolaWangEst} with \eqref{est prob} and the estimates on $\lVert \nabla v_s\rVert_{W^{m-1,\infty}}$ given by the induction step, after defining $A:= 2mC_b$ we can bound the integrand on each by the quantities: 
\begin{align}
    &  A c_{\delta/2}e^{-\alpha(T-s)}\frac{a}{\sqrt{s-t}}|x-y| \lVert \varphi\rVert_{W^{m-1,\infty}}, \quad  s \leq t+\delta/2 \leq T - \delta/2, \\
    & A c_{\delta/2}k_{\delta/2}e^{-\alpha(T-s)}e^{-\alpha(s-t)}|x-y|\lVert \varphi\rVert_{W^{m-1,\infty}}, \quad T -\delta/2> s > t+\delta/2, \\
    & A k_{\delta/2}\frac{c}{\sqrt{T-s}}e^{-\alpha(s-t)}|x-y|\lVert \varphi\rVert_{W^{m-1,\infty}}, \quad  s \ge T -\delta/2 \ge t+\delta/2,
\end{align} 
where, for the last two intervals, we bounded the integrand by $A \lVert \nabla v_s\rVert_{W^{m-1,\infty}}p_s(x,y)$ and then used the estimate \eqref{est prob} as in the base case. 
Integrating the previous quantities yields the desired bound on $\lVert \nabla^{m+1}v_t\rVert_\infty$, up to changing the exponent of the exponential decay from $\alpha$ to $\alpha/2$ and choosing constants correctly by taking $\delta > 0$ small enough. 
This concludes the proof of the induction, and therefore the proof of the estimates on the parabolic PDE.

\paragraph{Elliptic PDE.}
By \cite[Theorem 1.4]{WeakConc} there exists a Lipschitz-continuous transport map from the standard Gaussian $\gamma$ on $\R^d$ to $\mu$, so that $\mu$ satisfies a Poincaré inequality. Therefore, the Lax-Milgram theorem gives that \eqref{system elliptic} has a unique solution $\psi$ in $H^1(\mu)$.

Now, we set $\varphi_t = S_t\varphi$ and we prove that the function $\psi:= -\int_0^\infty \varphi_t \d t$ solves \eqref{system elliptic}. First, we prove that it is well defined pointwise and that it integrates to zero with respect to $\mu$. Using the boundedness of $\varphi_t$ and $\nabla\varphi_t$ for $t>0$ and the divergence theorem, we have \begin{equation*}
    \frac{\d}{\d t} \int_{\R^d}\varphi_t\d\mu = \int_{\R^d}\partial_t\varphi_t\d\mu = \frac{1}{2}\int_{\R^d}\Delta_\mu\varphi_t\d\mu = \frac{1}{2}\int_{\R^d}\nabla \cdot (\mu\varphi_t)\d x = 0
\end{equation*} where we also used the fact that $\partial_t\varphi_t = \frac{1}{2}\Delta_\mu\varphi_t$ is uniformly bounded for $t > 0$ by \eqref{estimates parabolic PDE}.
Since $\int_{\R^d}\varphi\d\mu = 0$, $\lVert\varphi_t\rVert_{\infty} \leq \lVert\varphi\rVert_\infty$ and $\varphi_t \to \varphi$ pointwise -- using the representation formula \eqref{Feynman Kac} and dominated convergence --, we have that $\int_{\R^d}\varphi_t\d\mu = 0$ for every $t \geq 0$. 
Therefore, integrating with respect to $\d\mu(y)$ and using the bounds on $\lVert \nabla \varphi_t\rVert_\infty$ in the formula \begin{equation*}
    \varphi_t(x) = \varphi_t(y) + \int_0^1 \nabla\varphi_t((1-s)y + s x)\cdot (x-y)\d s,
\end{equation*} we get \begin{equation*}
    |\varphi_t(x)| \leq \lVert\nabla \varphi_t\rVert_\infty \int_{\R^d}|x-y|\d\mu(y),
\end{equation*} where $\int_{\R^d}|y|\d\mu(y) < +\infty$ since $\mu$ has finite second moment, so that $\psi$ is well-defined pointwise and $\varphi_t(x) \in L^1(\d t \otimes \d\mu(x))$.
Therefore, we can write  \begin{equation*}
    \int_{\R^d}\psi\d\mu = -\int_{\R^d}\int_0^{\infty}\varphi_t\d t \d\mu = \int_0^\infty\int_{\R^d}\varphi_t\d\mu\d t = 0
\end{equation*} 
since $\varphi_t$ integrates to zero with respect to $\mu$ for every $t\geq 0$. Furthermore, using the exponential decay of $\nabla\varphi_t$ as $t \to +\infty$ we also get that $\varphi_t \rightarrow 0$ pointwise as $t \to +\infty$.

Now, by the estimates \eqref{estimates parabolic PDE} we can change the order of the time integral and the space derivatives $\nabla^{j+1}$, $0 \leq j \leq n$, to get the estimates \begin{equation*}
    \lVert \nabla^{j+1}\psi\rVert_\infty \leq C \lVert \varphi\rVert_{W^{j,\infty}}, \quad \forall \ 0\leq j \leq n.
\end{equation*} 
It only remains to prove that $\frac{1}{2}\Delta_\mu\psi = \varphi$. 
By the estimates on the parabolic PDE \eqref{estimates parabolic PDE}, we can change the order of the time integral and $\Delta_\mu$ to get \begin{equation*}
    \frac{1}{2}\Delta_\mu\psi = -\int_0^{+\infty}\frac{1}{2}\Delta_\mu\varphi_t\d t = -\int_0^{\infty}\partial_t \varphi_t\d t = -\lim_{t \to +\infty}\varphi_t + \varphi = \varphi.
\end{equation*} This concludes the proof.

\end{proof}

\begin{proof}[Proof of the extended Dacorogna-Moser construction (Theorem \ref{Theorem Dacorogna-Moser})]
Let $\psi$ be the unique solution of $\Delta_\mu \psi = \varphi$ such that $\int_{\R^d}\psi\d\mu = 0$ given by Proposition \ref{Lemma PDE estimates} (the result for the operator $\frac{1}{2}\Delta_\mu$ implies the one for $\Delta_\mu$) and define \begin{equation*}
    \rho_s:= (1+s\varphi)\cdot\mu, \quad v_s:= -\frac{\nabla\psi}{1+s \varphi}, \quad s \in [0,1].
\end{equation*} The couple $(\rho_s,v_s)_{s \in [0,1]}$ satisfies the continuity equation \eqref{eq:transportDM} so that, defining the flow of diffeomorphisms \begin{equation*}
    \begin{cases}
        \partial_sX_s = v_s \circ X_s, \\
        X_0 = \text{Id},
    \end{cases}
\end{equation*} 
we classically have that $(X_s)_{\#}\mu = \mu_s$, for every $s \in [0,1]$. Using standard Gronwall-type arguments to bound derivatives of the flow, the estimates on $\lVert\nabla\psi\rVert_{W^{n,\infty}}$ given by Proposition \ref{Lemma PDE estimates}, the fact that $|1+s\varphi| \geq \delta$ for $\delta<1$ -- since $\varphi \geq -1+\delta$ -- and the fact that $\lVert \varphi\rVert_{W^{n,\infty}} \leq M$, we can deduce the desired bounds on $\lVert X_1- \text{Id}\rVert_{W^{n,\infty}}$. To get the estimate for $\lVert X_1^{-1} - \text{Id}\rVert_{W^{n,\infty}}$ we just have to notice that $X^{-1}_1 = Y_0$ where \begin{equation*}
    \begin{cases}
        \partial_s Y_s = v_s \circ Y_s, \\
        Y_1 = \text{Id},
    \end{cases}
\end{equation*} by the uniqueness of the solution to the flow, which is given by the Lipschitz continuity of the velocity field $v_s$ uniformly with respect to $s \in [0,1]$. So we can repeat the previous argument to get the estimates for $X_1^{-1}$ and conclude the proof of this Theorem.
     
 \end{proof}

\subsection{The problem of geodesics in the space of couplings} \label{proofs geodesic}

\begin{proof}[Proof of the existence of geodesics and of $\Gamma$-convergence (Theorem \ref{Existence of geodesics})] We first prove the existence of geodesics and then $\Gamma$-convergence.

\paragraph{Existence of geodesics.} Assuming that $\mu$ and $\nu$ are $\lambda$-log-concave for some $\lambda > 0$, the existence of geodesics is proved in \cite[Section 3]{LouisLacker}, recalling that $\nabla^2\log\mu$ and $\nabla^2\log\nu$ are bounded from Assumption \ref{Assumption 1}. 
More specifically, \cite[Proposition 3.6]{LouisLacker} proves existence of curves of finite length in $\Pi(\mu,\nu)$ and uses the compactness criteria \cite[Proposition 3.1]{LouisLacker} to deduce existence of geodesics -- taking $\sigma$ therein to be the topology of weak convergence of measures. 
The same argument provides existence of geodesics in the strictly asymptotically log-concave setting if we show existence of curves of finite length connecting any two points in $\Pi( \mu, \nu)$. 

Using the strict asymptotic log-concavity of $\mu$, $\nu$ and Remark \ref{rem:profile}, \cite[Theorem 1.4]{WeakConc} provides invertible transport maps $T$ and $S$ with finite Lipschitz constant $L>0$ from the standard Gaussian $\gamma$ on $\R^d$ to $\mu$ and $\nu$, respectively. 
Let $\rho_0$, $\rho_1$ be any two points in $\Pi(\mu,\nu)$.
We have that $\eta_i:=(T^{-1},S^{-1})_{\#}\rho_i \in \Pi(\gamma,\gamma)$ for $i \in \{ 0,1 \}$. 
From the aforementioned result in \cite[Section 3]{LouisLacker}, there exists an absolutely continuous curve of finite length $(\eta_t)_{t \in [0,1]} \subset \Pi(\gamma,\gamma)$ connecting $\eta_0$ to $\eta_1$, and let $(w_t)_{t \in [0,1]}$ be a velocity vector field provided by \cite[Theorem 8.3.1]{ambrosioGradientFlowsMetric2008} such that $(\eta,w)$ satisfies the continuity equation \eqref{eq:transportDM} and 
\[ \mathcal{A}_0(\eta) = \int_0^1\int_{\R^d \times \R^d} \frac12 |w_t|^2\d\eta_t\d t. \] 
Defining \begin{equation*}
    \rho_t:= (T,S)_{\#}\eta_t, \quad v_t := \begin{pmatrix}
        \nabla T & 0\\
        0 & \nabla S
    \end{pmatrix}\cdot w_t \circ (T^{-1},S^{-1}),
\end{equation*} 
we can check that $(\rho,v)$ satisfies the continuity equation, and that $\rho_t \in \Pi(\mu,\nu)$ for every $t \in [0,1]$. 
From $|\nabla T|, |\nabla S| \leq L$, we have \begin{equation} \label{bound kinetic energy by Lipschitz maps}
    \mathcal{A}_0(\rho) \leq \int_0^1\int_{\R^d \times \R^d} \frac12 \bigg|\begin{pmatrix}
        \nabla T & 0\\
        0 & \nabla S
    \end{pmatrix} \cdot w_t \bigg|^2\d\eta_t \d t \leq L^2 \mathcal{A}_0(\eta) < +\infty,
\end{equation} so that we deduce that $(\rho_t)_{t \in [0,1]}$ is a curve of finite length in $\Pi(\mu,\nu)$ between $\rho_0$ and $\rho_1$, as desired.

\paragraph{Finiteness of the regularized problem.}
This part relies on estimates from \cite{GenMetricSpace}, which studies a general minimization problem in metric measure spaces. 
Rather than providing an incomplete introduction to the theory of optimal transport and gradient flow in metric spaces, we refer the reader to the textbook \cite{ambrosioGradientFlowsMetric2008} -- or more concisely to the definitions provided in \cite[Setting 3.1]{GenMetricSpace} and references therein. 

First, we notice that the minimization problem \eqref{def regularization} is equivalent to minimizing \begin{equation*}
\inf_{\rho \in AC([0,1],\mathcal{P}(\R^d \times \R^d))}    \bigg\{ \int_0^1|\dot\rho|^2_t\d t + \frac{\varepsilon^2}{8}\int_0^1\int_{\R^d \times \R^d}\bigg |\nabla\log\frac {\d\rho_t}{\d \mu \otimes\nu}\bigg|^2\d\rho_t\d t, \ (\rho_t)_{t \in [0,1]} \subset \Pi(\mu,\nu)\bigg\},
\end{equation*} 
where the endpoints are still $\rho_0, \rho_1$ and we recall that the metric derivative $|\dot\rho|_t$ of an absolutely continuous curve $(\rho_t)_{t \in [0,1]}$ is defined as
\begin{equation*}
    |\dot \rho|_t:= \limsup_{s \to t}\frac{\mathcal{W}_2(\rho_t,\rho_s)}{|t-s|}.
\end{equation*} 
By the boundedness of $\nabla^2\log\mu$ and $\nabla^2\log\nu$, the functional
$E(\rho):=\frac12 H(\rho|\mu\otimes\nu)$ is $\lambda$-convex along geodesics
in $(\mathcal{P}_2(\R^d\times\R^d),\mathcal{W}_2)$ for some $\lambda<0$, and therefore the $EVI_{\lambda}$-gradient flow $( S_t\gamma )_{t \geq 0}$ of $\rho \mapsto H(\rho|\mu\otimes\nu)$ starting from any $\gamma \in \mathcal{P}_2 ( \R^d \times \R^d)$ is well-defined -- see e.g. \cite[Setting 3.1]{GenMetricSpace}.
By studying the associated Fokker-Planck equation, $(S_t)_{t \geq 0}$ leaves invariant the space of couplings $\Pi(\mu,\nu)$  -- see e.g. \cite[Proof of Theorem 4.6]{LouisLacker}.

Now, let $\rho_0, \rho_1 \in \Pi(\mu,\nu)$ with finite relative entropy, and let $(\rho_t)_{t\in[0,1]}\subset\Pi(\mu,\nu)$ be the curve constructed
in the first part of the proof, which belongs to $AC^2$ and connects them.
We define $\rho^\varepsilon_t:=S_{\frac{\varepsilon}{2}\min\{t,1-t\}}\rho_t$. We notice that $(\rho^{\varepsilon}_t)_{t \in [0,1]} \subset \Pi(\mu,\nu)$. Since $(\mathcal{P}_2(\R^d \times \R^d), \mathcal{W}_2)$ satisfies \cite[Setting 3.1, Assumption 3.2]{GenMetricSpace}, we can use the estimate \cite[Theorem 3.12]{GenMetricSpace} -- applied with parameter $\delta = \varepsilon/2$ -- combined with $H(\cdot|\mu\otimes\nu) \geq 0$ and $\lambda < 0$, to deduce that
\begin{equation} \label{estimate for Gamma convergence}
\int_0^1|\dot \rho^\varepsilon|^2_t\d t
+ \frac{\varepsilon^2}{4}
\int_0^1|\partial E(\rho^\varepsilon_t)|^2\d t
\leq
e^{-\lambda\varepsilon/2}
\int_0^1|\dot \rho|^2_t\d t
+\frac{\varepsilon}{2}
\big(
H(\rho_0|\mu\otimes\nu)
+H(\rho_1|\mu\otimes\nu)
\big).
\end{equation} 
where 
\begin{equation*}
    |\partial E(\gamma)|^2
    = \frac{1}{2}
    \int_{\R^d\times\R^d}
    \bigg|\nabla\log\frac{\d\gamma}{\d\mu\otimes\nu}\bigg|^2
    \d\gamma,
    \quad
    \forall\gamma:\ H(\gamma|\mu\otimes\nu)<+\infty,
\end{equation*}
from \cite[Theorem 10.4.9]{ambrosioGradientFlowsMetric2008},
after accounting for the normalization \eqref{eq:BB}. Since $\rho^{\varepsilon}_t \in \Pi(\mu,\nu)$ for every $t \in [0,1]$ and $ \varepsilon \in [0,1]$, \eqref{estimate for Gamma convergence} implies that \eqref{def regularization} is finite. 

\paragraph{Existence and uniqueness for the regularized problem.} 
The existence of a minimizer is a direct adaptation of \cite[Proposition 3.1]{LouisLacker}, using compactness for a sequence of minimizing curves and lower semicontinuity with respect to the weak convergence of measures.
The lower semicontinuity of the objective function follows from the one of $\mathcal{A}_0(\rho)$ -- see e.g. \cite[Section 2.2]{AGS13Ricci} -- and the one of the remaining term, using Fatou's lemma and the lower semicontinuity of $|\partial H(\cdot|\mu\otimes\nu)|^2$ -- proved in \cite[Theorem 10.4.14]{ambrosioGradientFlowsMetric2008}. 
The uniqueness of the minimizer for $\varepsilon > 0$ is a consequence of the strict convexity of the Fisher information on the set of positive probability measures, combined with the usual convexity of the kinetic action with respect to $(\rho,\rho v)$ for $(\rho,v)$ satisfying the continuity equation \eqref{eq:transportDM}.  

\paragraph{$\boldsymbol{\Gamma}$-convergence.} 
Since $d_{\Pi(\mu,\nu)}$-uniform convergence is stronger than pointwise-in-time convergence with respect to the weak convergence of measures, it is sufficient to prove the $\lim\inf$ inequality with respect to the pointwise convergence -- so it is a consequence of the lower semicontinuity of the kinetic energy $\mathcal{A}_0$ and the non-negativity of the Fisher information. Similarly, it is sufficient to prove the $\limsup$ inequality with respect to the $d_{\Pi(\mu,\nu)}$-uniform one. Let $(\rho_t)_{t \in [0,1]} \subset \Pi(\mu,\nu)$ be a curve such that $\mathcal A_0(\rho)<+\infty$ connecting $\rho_0$ and $\rho_1$ and define $\rho^\varepsilon_t
:=S_{\frac{\varepsilon}{2}\min\{t,1-t\}}\rho_t$ as above. If we prove that $\sup_{t \in [0,1]}d_{\Pi}(\rho_t^\varepsilon,\rho_t) \to 0$ as $\varepsilon \to 0$, the $\limsup$ inequality is a consequence of the estimate \eqref{estimate for Gamma convergence} and the finiteness of the relative entropy of $\rho_0$ and $\rho_1$. 
Since the curve $(S_s\rho_t)_{s \in [0,\varepsilon t (1-t)]} \subset \Pi(\mu,\nu)$ connects $\rho_t$ and $\rho^{\varepsilon}_t$, we have that \begin{equation} \label{uniform convergence recovery sequence}
    d_{\Pi}(\rho_t^\varepsilon,\rho_t)
\leq
\int_0^{\frac{\varepsilon}{2}\min\{t,1-t\}}
|\dot S_s\rho_t|_s\d s.
\end{equation} 
We recall the energy identity \cite[Theorem 3.5]{ExpEstimate}
\begin{equation} \label{energy identity}
    |\dot S_s\rho_t|_s^2
=
|\partial E(S_s\rho_t)|^2
=
\frac{1}{2}
\int_{\R^d\times\R^d}
\bigg\lvert
\nabla\log\frac{\d S_s\rho_t}{\d\mu\otimes\nu}
\bigg\rvert^2
\d(S_s\rho_t).
\end{equation} Using the stochastic representation formula 
\begin{align} \label{Bismut type}
    \nabla\log S_s(\rho_t)(x) & = -\frac{1}{\sqrt{2}s}\mathbb{E}\bigg[\int_0^s(I - r \nabla^2\log\mu\otimes\nu(X_r))\d B_r \bigg| X_s = x\bigg ], \\
    \d X_r & = \nabla\log\mu\otimes\nu(X_r) \d r + \nonumber
    \sqrt{2}\d B_r, \quad X_0 \sim \rho_t,
\end{align} which we got from \cite[Remark 4.13]{BismutStuff}, we now show that \begin{equation} \label{ineq short time gamma convergence}
    |\dot S_s\rho_t|_s \leq c(1+1/\sqrt{s})
\end{equation} for some $c>0$.
To do so, we square \eqref{Bismut type} and we use It{\^o}'s isometry as well as the boundedness  of $\nabla^2\log\mu\otimes\nu$. 
Integrating with respect to $\d (S_s\rho_t)$,
\begin{equation} \label{first part ineq short time}
    \int_{\R^d \times \R^d}\big\lvert \nabla\log S_s\rho_t\big\rvert^2\d (S_s\rho_t) \le \frac{a}{s}
\end{equation} 
for some constant $a>0$ and every $0<s\leq1/4$. Furthermore, by the boundedness of $\nabla^2\log\mu\otimes\nu$, the term \begin{equation*}
    \int_{\R^d \times \R^d}|\nabla\log\mu\otimes\nu|^2\d (S_s\rho_t)
\end{equation*} is bounded uniformly for $ s \geq 0$, so by \eqref{first part ineq short time} and the energy identity \eqref{energy identity} we deduce \eqref{ineq short time gamma convergence}. Therefore, the right-hand side of \eqref{uniform convergence recovery sequence} is bounded uniformly in $t \in [0,1]$ by a quantity that goes to zero as $\varepsilon \to 0$. This concludes the proof.

\end{proof}

\begin{prop}[Convexity and lower semicontinuity of $\mathcal{F}_{\varepsilon}$] \label{Prop convexity and lower semicontinuity}
For every $\varepsilon \in [0,1]$ the functional $\mathcal{F}_{\varepsilon}$ is convex and lower semicontinuous with respect to the topology on $C([0,1],\mathcal{P}_2(\R^d))^2$. 
\end{prop}

\begin{proof}
For the convexity, if $\rho^i_t \in \Pi(\mu^i_t, \nu^i_t)$ for every $ t \in [0,1], i \in \{0,1\}$, then $(1-s)\rho^0_t+s\rho^1_t \in \Pi(\mu^s_t,\nu^s_t)$ for every $t, s \in [0,1]$, where $\mu^s_t:= (1-s)\mu_t^0 + s\mu_t^1$ and $\nu^s_t:= (1-s)\nu^0_t + s\nu^1_t$.
Leveraging the convexity of the relative Fisher information, we can then consider convex combinations of minimizers and reproduce \cite[Proof of Lemma 12]{Baradatoptimality} without any change.

For the lower-semicontinuity, if $\rho^{(n)} \in \Pi(\mu^{(n)},\nu^{(n)})$, $\rho^{(n)} \to \rho$, $\mu^{(n)} \to \mu$ and $\nu^{(n)} \to \nu$ for the weak convergence of measures, then it is direct that $\rho \in \Pi(\mu,\nu)$.
Since $\mathcal{A}_{\varepsilon}$ is lower-semicontinuous with respect to the pointwise-in-time convergence, we can then substitute $\boldsymbol {\mathcal{H}}_\nu$ with $\mathcal{A}_{\varepsilon}$ in  \cite[Proof of Lemma 12]{Baradatoptimality}, and the result follows from the exact same arguments.
\end{proof}

\begin{proof}[Proof of existence for the Lagrange multiplier (Theorem \ref{Existence of the Lagrange multiplier})]

By convexity and lower semicontinuity of $\mathcal{F}_{\varepsilon}$ on $S_\mu \times S_\nu$, the convex analysis result \cite[Theorem 23.4]{Rockafellar} reduces the problem of existence for a Lagrange multiplier, i.e. an element $p^\varepsilon \in \partial\mathcal{F}_{\varepsilon}(0,0)$, to proving that $\mathcal{F}_{\varepsilon}$ is bounded in a neighborhood of $(0,0)$ in $S_\mu \times S_\nu$. 
Let $(\rho,v)$ be a solution of \eqref{def regularization} (which exists by Theorem \ref{Existence of geodesics}), and let $\varphi \in S_\mu, \psi \in S_\nu$ such that $N(\varphi), N(\psi) \leq 1/2$. 
For every $t \in [0,1]$, let $T_t$ and $S_t$ be the transport maps constructed in Theorem \ref{Theorem Dacorogna-Moser} such that \begin{equation*}
    (T_t)_{\#}\mu = (1+\varphi_t)\cdot\mu, \quad (S_t)_{\#}\nu = (1+\psi_t)\cdot\nu,
\end{equation*} 
and define -- the time-regularity of $T_t$, $S_t$ being immediate from the construction in Theorem \ref{Existence of geodesics} -- 
\begin{equation*}
    \tilde{\rho}_t := (T_t, S_t)_{\#}\rho_t, \quad \tilde{v}_t:= \bigg(\partial_t(T_t, S_t)  + \begin{pmatrix}
        \nabla T_t & 0 \\
        0 & \nabla S_t
    \end{pmatrix} \cdot v_t\bigg) \circ (T_t^{-1},S_t^{-1}).
\end{equation*} 
We can check that $\tilde{\rho}_t \in \Pi((1+\varphi_t)\cdot\mu,(1+\psi_t)\cdot\nu) $ for every $t \in [0,1]$, and $(\tilde{\rho}, \tilde{v})$ satisfies the continuity equation, so that it is a candidate for the minimization problem $\mathcal{F}_{\varepsilon}(\varphi,\psi)$.
From there, we closely follow  \cite[Proof of Lemma 13]{Baradatoptimality}.

\paragraph{Bounding $\boldsymbol{\mathcal{A}_{0}(\tilde{\rho})}$.} 
Writing $T_t(x) = x +\xi_t(x)$, $S_t(y) = y + \zeta_t(y)$, a straightforward adaptation of \cite[Proof of Lemma 13, Step 2]{Baradatoptimality} yields \begin{equation*}
    \mathcal{A}_0(\tilde{\rho}) \lesssim \bigg( 1 + \frac{1}{2}\int_0^1\int_{\R^d \times \R^d}|v_t|^2\d\rho_t\d t\bigg)(1 + N(\xi)^2 + N(\zeta)^2),
\end{equation*} 
where $N(\xi) \le  CN(\varphi), N(\zeta) \le CN(\psi) $ for some $C>0$ by the estimates of Theorem \ref{Theorem Dacorogna-Moser}. 
In particular, this proves the bound for $\varepsilon = 0$.

\paragraph{Bounding the relative Fisher information.} 
To alleviate notations, set $R_t(x,y) := (T_t(x),S_t(y))$, $T_t^{-1}(x) = x + \tilde{\xi}_t(x)$, $S_t^{-1}(y) = y + \tilde{\zeta}_t(y)$ and \begin{equation*}
    \mathcal{G}(\gamma):= \frac{1}{8}\int_0^1\int_{\R^d\times\R^d}\bigg|\nabla\log\frac{\d \gamma_t}{\d \mu\otimes\nu}\bigg|^2\d\gamma_t\d t.
\end{equation*} 
By definition of $\tilde{\rho}$, 
\[ \nabla\log\tilde{\rho}_t  =^t\!\nabla R^{-1}_t \cdot \nabla\log\rho_t \circ R_t^{-1} + F_t \quad \text{for} \quad F_t := (\nabla_x\log\det\nabla T_t^{-1}, \nabla_y\log\det\nabla S_t^{-1}), \]
so that, using $(R_t)_{\#}\rho_t = \tilde{\rho}_t$, 
\begin{align*}
    \mathcal{G}(\tilde{\rho}) &= \frac{1}{8}\int_0^1\int_{\R^d\times\R^d}\bigg|\nabla\log\frac{\d \tilde{\rho}_t}{\d \mu\otimes\nu} \circ R_t\bigg|^2\d \rho_t\d t \\
    & = \frac{1}{8}\int_0^1\int_{\R^d\times\R^d}\bigg| \  ^t\! \nabla R^{-1}_t (R_t) \cdot \bigg[\nabla\log\frac{\d \rho_t}{\d \mu \otimes \nu} + \nabla\log(\mu\otimes\nu)\bigg]+ F_t \circ R_t - \nabla\log(\mu\otimes\nu) \circ R_t\bigg|^2\d\rho_t \d t.
\end{align*} 
Using the estimates on $R_t^{-1}$ given by Theorem \ref{Theorem Dacorogna-Moser}, we bound the first term by \begin{equation*}
   C_0(1+ N(\varphi)^2 + N(\psi)^2)\mathcal{G}(\rho),
\end{equation*} for some constant $C_0 >0$. 
Since $\nabla\log\mu$ and $\nabla\log\nu$ have at most linear growth and $\rho_t \in \Pi(\mu,\nu)$, we bound the terms involving $\nabla\log\mu\otimes\nu$ by \begin{equation*}
    C_1(1+ N(\varphi)^2+N(\psi)^2),
\end{equation*} for some constant $C_1 > 0$.
To bound $F_t \circ R_t$, we use the identity \begin{equation*}
    \nabla_x\log\det \nabla T_t^{-1} \circ T_t(x) = - (\nabla T_t^{-1} \cdot \nabla\det \nabla T_t\circ T^{-1})\det \nabla T_t^{-1},
\end{equation*} where $\lVert\nabla T_t^{-1}\rVert_{\infty}$ is bounded by a numerical constant that does not depend on $\varphi$ by Theorem \ref{Theorem Dacorogna-Moser} and the fact that $N(\varphi)\le 1/2$. Furthermore, $\det \nabla T_t^{-1}$ and $\nabla\det \nabla T_t$ are linear combinations of numerical constants and first and second order derivatives of $\tilde{\xi}_t$, which can be bounded by a numerical constant independent of $\varphi$, since $N(\varphi) \le 1/2$. Therefore, $\lVert F_t \rVert_{\infty}$ is uniformly bounded by a numerical constant. 

Gathering the bound on $\mathcal{G}(\tilde{\rho})$ with the previous one on $\mathcal{A}_0(\tilde{\rho})$ and using that $N(\varphi), N(\psi) \leq 1/2$, we get 
\begin{equation} \label{bound for the functional}
    \mathcal{F}_{\varepsilon}(\varphi,\psi) \leq \mathcal{A}_{\varepsilon}(\tilde{\rho})  \leq C(1+\mathcal{A}_{\varepsilon}(\rho)) = C(1+ \mathcal{F}_{\varepsilon}(0,0)) \leq C(1+\mathcal{F}_1(0,0)),
\end{equation} 
where $C>0$ is a constant that does not depend on the choice of $\varphi \in S_\mu, \psi \in S_\nu$ and $\varepsilon \in [0,1]$. 
This concludes the proof.
\end{proof}

\begin{proof}[Proof of optimality conditions and uniqueness of the Lagrange multiplier (Theorem \ref{Optimality condition for the Lagrange multiplier})] 

To derive optimality conditions, we can perturb one marginal at a time using the linearity of $p^{\varepsilon}$. 
Let $\varphi \in \tilde{S}_\mu $ and let $f$ be the solution of \begin{equation*}
        \Delta_\mu f = \varphi , \quad
        \int_{\R^d}f \d\mu = 0.
\end{equation*} 
Define $T^\delta_t(x) := x + \delta \nabla f_t(x)$, and \begin{equation*}
    \varphi^\delta_t(x) := \vert \det \nabla T^\delta_t(x)
    \vert \frac{\mu(T_t^\delta(x))}{\mu(x)} - 1.
\end{equation*} 
We now prove that $\varphi^\delta \in S_\mu$, so that we can write the dual inequality $\mathcal{F}_{\varepsilon}(\varphi^\delta,0) \geq \mathcal{F}_{\varepsilon}(0,0) + \langle p^\varepsilon, (\varphi^\delta,0)\rangle$, and compute a first-order expansion to derive the optimality conditions. 

\paragraph{$\boldsymbol{\varphi^\delta}$ belongs to $\boldsymbol{S_\mu}$.} 
By Proposition \ref{Lemma PDE estimates}, $\nabla f_t(x) \in L^\infty([0,1], W^{3,\infty})\cap W^{1,2}([0,1],W^{1,\infty})$ so that $\nabla T^\delta_t(x) = \mathrm{Id} + \delta \nabla^2f_t(x)$ is uniformly bounded; for $\delta > 0$ small enough this implies that 
\begin{equation} \label{eq:expDet}
\vert \det\nabla T^\delta_t(x) \vert = \det \nabla T^\delta_t(x) = 1+ \delta \Delta f_t(x) + \delta r_t^\delta(x), 
\end{equation} 
where $r_t^\delta(x)$ is a function in $L^\infty([0,1],W^{2,\infty}) \cap W^{1,2}([0,1],L^\infty(\R^d))$ such that $N(r^\delta) \to 0$ as $\delta \downarrow 0$. 
We now prove that $\mu\circ T^\delta_t(x)/\mu(x)$ is uniformly bounded in time and space. Indeed, defining  \begin{equation*}
    R_t^\delta(x):= \sup_{y \in [x, x+\delta\nabla f_t(x)]} \frac{\mu(y)}{\mu(x)},
\end{equation*} we have
\begin{align*}
    |R_t^\delta(x)-1| & \leq \delta\sup_{y \in [x, x+\delta\nabla f_t(x)]} \frac{\nabla\mu(y)}{\mu(x)} \cdot \nabla f_t(x) \leq \delta \sup_{y \in [x, x+\delta\nabla f_t(x)]}\nabla\log\mu(y)\cdot\nabla f_t(x) R_t^\delta(x) \\[0.2cm]
    & \leq \delta R_t^\delta(x)( \delta \lVert\nabla^2\log\mu\rVert_\infty
   \sup_{t \in [0,1]} \lVert\nabla f_t\rVert_\infty + |\nabla\log\mu(x) \cdot \nabla f_t(x)|) \\[0.2cm]
    & \leq  \delta R_t^\delta(x)( \delta \lVert\nabla^2\log\mu\rVert_\infty \sup_{t \in [0,1]} \lVert\nabla f_t\rVert_\infty + \sup_{t \in [0,1]}\lVert \varphi_t - \Delta f_t\rVert_\infty),
\end{align*} 
where we used $\nabla\log\mu\cdot\nabla f_t = \Delta_\mu f_t - \Delta f_t = \varphi_t - \Delta f_t$, together with our assumptions on $\varphi$ and the bounds from Proposition \ref{Lemma PDE estimates}.

Therefore, there exists a constant $C>0$ that does not depend on $t$ and $x$ such that $R_t^\delta(x) \leq 1 + C\delta R_t^\delta(x)$, so that $R_t^\delta(x) \leq 1/(1-C\delta)$ for $\delta$ small enough, which implies the uniform bound on $\mu\circ T^\delta_t(x)/\mu(x)$ and the fact that $\sup_t\lVert\varphi_t^\delta\rVert_\infty < +\infty$.

We now prove the uniform boundedness of $\nabla\varphi^\delta$ and of $\nabla^2\varphi^\delta$. 
Since the quantities on the right-hand side of \eqref{eq:expDet} are in $L^\infty([0,1],W^{2,\infty})$ from $\varphi \in S_\mu \cap L^\infty([0,1],W^{3,\infty})$ and Proposition \ref{Lemma PDE estimates}, we only need to prove the boundedness of the first and second derivatives of $\mu\circ T^\delta_t(x)/\mu(x)$. We have \begin{equation} \label{derivative nabla mu frac}
    \nabla\frac{\mu\circ T^\delta_t(x)}{\mu(x)} = ((\mathrm{Id} + \delta \nabla^2 f_t)\cdot\nabla\log\mu \circ T^\delta_t(x) - \nabla\log\mu(x))\frac{\mu\circ T^\delta_t(x)}{\mu(x)}, 
\end{equation} 
where we saw that $\mu\circ T^\delta_t(x)/\mu(x)$ is uniformly bounded, and 
\[ |\nabla\log\mu\circ T_t^\delta -\nabla\log\mu| \leq \delta \lVert \nabla^2\log\mu\rVert_\infty \sup_{t \in [0,1]}\lVert \nabla f_t\rVert_\infty. \] 
Thus, we only have to bound $\nabla^2 f_t\cdot\nabla\log\mu \circ T^\delta_t(x)$ in \eqref{derivative nabla mu frac}. Using the expansion 
\[ \nabla\log\mu\circ T^\delta_t(x) = \nabla\log\mu(x) + \delta\nabla^2\log\mu(\hat{x})\cdot\nabla f_t(x), \]
where $\hat{x}$ is given by the Taylor-Laplace formula, and the fact that $\nabla^2\log\mu$, $\nabla f$, $\nabla^2 f$ are uniformly bounded, it only remains to prove that $\nabla^2 f_t \cdot\nabla\log\mu$ is uniformly bounded. This follows from differentiating
\begin{equation} \label{eq:ToDiff}
\Delta f_t + \nabla\log\mu \cdot \nabla f_t = \varphi_t, 
\end{equation} 
and combining the bounds from Proposition \ref{Lemma PDE estimates} with the fact that $\nabla^2\log\mu$ is bounded. 

To prove the bound on $\nabla^2\mu\circ T_t^\delta/\mu$, we differentiate \eqref{derivative nabla mu frac}. 
Using similar arguments, we end up bounding the quantity $\langle\nabla^3 f_t(x), \nabla\log\mu(x)\rangle$.
This follows from differentiating \eqref{eq:ToDiff} twice and using the boundedness of $\nabla^3\log\mu$. 

We thus prove that $\varphi^\delta \in L^\infty([0,1],W^{2,\infty})$. Using similar arguments, we can prove that $\varphi^\delta \in W^{1,2}([0,1],L^\infty(\R^d))$. 
It remains to prove that $\int_{\R^d}\varphi^\delta(y) \d\mu (y) = 0$, which follows from the change of variable $y = T^\delta_t(x)$, for $\delta$ small enough. Therefore, $\varphi^\delta$ belongs to $S_\mu$.

\paragraph{Expansion of $\boldsymbol{\langle p^\varepsilon, (\varphi^\delta,0)\rangle}$.}

For $\delta$ small enough, we have the expansion $|\det\nabla T^\delta_t(x)| = 1+ \delta\Delta f_t(x) + \delta r_t^\delta(x)$, where $N(r^\delta) \to 0$ as $\delta \downarrow 0$. 
Performing a Taylor expansion, we further have 
\begin{equation} \label{eq:RemainExpan}
    \frac{\mu \circ T^\delta_t(x)}{\mu(x)} = 1 + \delta\nabla\log\mu(x) \cdot \nabla f_t(x) + \int_0^\delta  \frac{(\delta-r)^2}{2} \frac{\nabla^2\mu\circ T_t^r(x)}{\mu(x)}\nabla f_t(x) \cdot \nabla f_t(x) \d r.
\end{equation} 
Let $\delta s^\delta_t(x)$ denote the last term on the right-hand side. 
From the Lipschitz-continuity of $\nabla^j\log\mu$ for $j \in \{1,2,3\}$, the same computations that proved that $\varphi^\delta \in S_\mu$ show that $N(s^\delta) \to 0$ as $\delta \downarrow 0$. Gathering \eqref{eq:expDet}-\eqref{eq:RemainExpan} yields
\begin{equation*}
    \varphi^\delta_t(x) = (1+\delta\Delta f_t + \delta r_t^\delta)(1+\delta \nabla\log\mu\cdot\nabla f_t + \delta s_t^\delta) -1 = \delta \Delta_\mu f_t + \delta z_t^\delta = \delta \varphi_t + \delta z_t^\delta,
\end{equation*} where $N(z^\delta) \to 0$ as $\delta \to 0$. Using the fact that $\varphi^\delta$ and $\varphi$ belong to $S_\mu$ we further get that $z^\delta \in S_\mu$, so that 
\begin{equation} \label{expansion dual brackets}
    \langle p^\varepsilon, (\varphi^\delta,0)\rangle = \delta \langle p^\varepsilon, (\varphi,0)\rangle + o(\delta).
\end{equation} 

\paragraph{Expansion of $\boldsymbol{\mathcal{F}_{\varepsilon}(\varphi^\delta,0)}$.} 
Let us now expand the left-hand side in the dual inequality. 
Let $(\rho, v)$ be a minimizer for \eqref{def regularization} and let $\rho_t^\delta := ((T_t^\delta)^{-1},\text{Id})_{\#}\rho_t$ and $v^\delta_t(x,y):= [ (\partial_t(T^\delta_t)^{-1},0)+ (\nabla_{x,y}(T_t^\delta)^{-1}, I_d) \cdot v_t] \circ (T_t^\delta(x),y)$ .
Since $(\rho^\delta,v^\delta)$ satisfies the continuity equation and $(T^\delta_t)_{\#}(1+\varphi^\delta)\cdot\mu = \mu$, this provides a competitor for $\mathcal{F}_{\varepsilon}(\varphi^\delta,0)$. 
Let us now expand the inequality $\mathcal{A}_{\varepsilon}(\rho^\delta) \geq \mathcal{F}_{\varepsilon}(\varphi^\delta,0)$. 

From $\sup_{t \in [0,1]}\lVert\nabla f_t\rVert_{W^{3,\infty}} < +\infty$, the map $T^\delta$ is a diffeomorphism for $\delta>0$ small enough and $\nabla^j T^\delta, \nabla^j (T^{\delta})^{-1}$ are uniformly bounded in time and space for $1\leq j \leq 3$. 
Therefore, a straightforward adaptation of
\cite[Proof of Lemma 14, Equation (39)]{Baradatoptimality} yields
the expansion 
\begin{equation} \label{term zero}
    \mathcal{A}_0(\rho^\delta) = \mathcal{A}_0(\rho)+\delta \big\langle\partial_t(\rho v) + \operatorname{div}(\rho(v \otimes v)), (\nabla f,0)\big\rangle + o(\delta),
\end{equation} where \begin{align*}
    \big\langle\partial_t(\rho v), (\nabla f,0)\big\rangle & = \int_0^1\int_{\R^d \times\R^d} (\nabla f_t,0)\cdot v_t\d \rho_t\d t \\
    \big\langle\operatorname{div}(\rho(v \otimes v)), (\nabla f,0)\big\rangle & = \int_0^1\int_{\R^d \times\R^d} \begin{pmatrix}
        \nabla^2f_t & 0\\
        0 & 0
    \end{pmatrix}v_t \cdot v_t \d \rho_t \d t.
\end{align*} This is well defined by the bounds from Proposition \ref{Lemma PDE estimates}.

We now expand $\mathcal{G}(\rho^\delta)$. Since $\rho_t^\delta = (T_t^{\delta})^{-1}_{\#}\rho_t$, we have \begin{equation*}
    \mathcal{G}(\rho^\delta) = \frac{1}{8}\int_0^1\int_{\R^d\times \R^d}\bigg|\nabla\log\frac{\d \rho^\delta_t}{\d \mu\otimes\nu} \circ ((T_t^{\delta})^{-1}(x),y)\bigg|^2\d\rho_t\d t,
\end{equation*} 
and we expand the right-hand side in this equality. Once again, we borrow the following equality from \cite[Proof of Lemma 14, Estimation of $\mathcal{F}(\rho^\varepsilon)$]{Baradatoptimality}:
\begin{equation*}
   \nabla_x\log\rho^\delta_t \circ ((T_t^\delta)^{-1}(x), y) = (I + \delta \nabla^2 f_t(x))\cdot \nabla_x\log\rho_t(x,y) + \delta \nabla \Delta f_t(x) + o(\delta).
\end{equation*} 
A Taylor expansion further yields
\begin{equation*}
    \nabla\log\mu(x) - \nabla\log\mu(T_t^\delta)^{-1}(x) = \delta\nabla^2\log\mu(x)\cdot\nabla f_t(x) + o(\delta),
\end{equation*} 
so we get \begin{align*}
    &\nabla_x\log\frac{\d \rho^\delta}{\d \mu\otimes\nu} ((T_t^\delta)^{-1}(x),y) \\
    & = (I+\delta ^t\! \ \nabla^2 f_t)\cdot  \nabla_x\log\frac{\d \rho_t}{\d \mu \otimes \nu}  + \delta\nabla\Delta f_t(x) + \delta^t\!\ \nabla^2 f_t \cdot \nabla\log\mu + \delta \nabla^2\log\mu\cdot\nabla f_t + o(\delta).
\end{align*} In particular, this allows us to write \begin{equation} \label{first step Fisher}
    \mathcal{G}(\rho^\delta)-\mathcal{G}(\rho) = \frac{\delta}{4}\int_0^1\int_{\R^d \times \R^d}\langle \nabla_x\log\frac{\d\rho_t}{\d \mu\otimes\nu}, ^t\! \ \nabla^2 f_t \cdot\nabla_x\log\rho_t + \nabla \Delta f_t + \nabla^2\log\mu\cdot\nabla f_t\rangle\d\rho_t\d t + o(\delta).
\end{equation} 
Using integration by parts and $\rho_t \in \Pi(\mu,\nu)$, we have
\begin{align}
    &\int_{\R^d \times \R^d}\nabla_x\log\frac{\d \rho_t}{\d \mu \otimes \nu}\cdot \nabla\Delta f_t \d \rho_t = 0,  \nonumber\\
    & \int_{\R^d \times \R^d}\nabla_x\log\frac{\d \rho_t}{\d \mu \otimes \nu}\cdot \nabla(\nabla\log\mu \cdot \nabla f_t)\d\rho_t = 0, \quad \forall t \in [0,1]. \label{intparts 2}
\end{align} Indeed, \begin{equation} \label{eq cancel out}
    \int_{\R^d \times \R^d}\nabla_x\log\frac{\d \rho_t}{\d \mu \otimes \nu}\cdot \nabla\Delta f_t \d \rho_t = \int_{\R^d \times \R^d}\nabla_x\log\rho_t\cdot \nabla\Delta f_t \d \rho_t - \int_{\R^d \times \R^d}\nabla_x\log\mu\cdot \nabla\Delta f_t \d \rho_t
\end{equation} and \begin{align*}
    \int_{\R^d \times \R^d}\nabla_x\log\rho_t\cdot \nabla\Delta f_t \d \rho_t & = \int_{\R^d \times \R^d}\nabla_x\rho_t\cdot \nabla\Delta f_t \d x \d y = \int_{\R^d \times \R^d} \Delta^2 f_t \rho_t\d x \d y \\
    & = \int_{\R^d} \Delta^2 f_t \d \mu(x) = \int_{\R^d} \nabla\log\mu \cdot \nabla f_t(x) \d \mu(x) \\
    & = \int_{\R^d \times \R^d} \nabla\log\mu \cdot \nabla f_t \d \rho_t,
\end{align*} where $\Delta^2$ means that the Laplacian is applied twice, so that the terms in \eqref{eq cancel out} cancel out. A similar computation can be done for the second equation in \eqref{intparts 2}. Therefore, using \eqref{first step Fisher} and $\nabla(\nabla\log\mu \cdot \nabla f_t) = \nabla^2\log\mu \cdot\nabla f_t + ^t\!
 \nabla^2 f_t \cdot\nabla\log\mu$ we can write \begin{equation} \label{Final expansion Fisher}
    \mathcal{G}(\rho^\delta) = \mathcal{G}(\rho) + \frac{\delta}{4}\int_0^1\int_{\R^d \times \R^d}\bigg\langle \nabla_x\log\frac{\d \rho_t}{ \d\mu\otimes\nu}, ^t\!  \nabla^2 f_t \cdot \nabla_x\log\frac{\d \rho_t}{\d \mu\otimes\nu}\bigg\rangle \d\rho_t \d t + o(\delta).
\end{equation}

\paragraph{Optimality conditions.} Gathering \eqref{expansion dual brackets}-\eqref{term zero}-\eqref{Final expansion Fisher}, dividing by $\delta$ and taking $\delta \to 0$, we get 
\begin{equation*}
    \langle p^\varepsilon, (\varphi,0)\rangle \leq \langle \partial_t(\rho v) + \operatorname{div}(\rho(v \otimes v - w \otimes w)), (\nabla f,0)\rangle,
\end{equation*} where we defined $w$ as in the statement of Theorem \ref{Optimality condition for the Lagrange multiplier}, so that substituting $\varphi$ with $-\varphi$ yields an equality. 
The same computations for perturbations $\psi \in S_\nu$ of the second marginal then yield the optimality conditions, using the linearity of $p^\varepsilon$.
For every $\varepsilon \in [0,1]$,
the choice of the minimizer $(\rho^{\varepsilon},v^{\varepsilon})$ in the right-hand side of the optimality conditions was arbitrary, so that we deduce uniqueness of the Lagrange multiplier $p^{\varepsilon}$ in $(\tilde{S}_\mu \times \tilde{S}_\nu)'$, concluding the proof.
\end{proof}

\begin{proof}[Proof of the convergence result Proposition \ref{Proposition weak star convergence}]

Using the optimality conditions \eqref{optimality conditions}, we only need to prove that the right-hand side is continuous at $\varepsilon = 0^+$. 
By linearity of $p^\varepsilon$, we can perturb one marginal at a time.
Let $(\rho^\varepsilon, v^\varepsilon)$ be a minimizer for $\mathcal{F}_{\varepsilon}(0,0)$, $\varphi \in \tilde{S}_\mu$ and $ f$ the solution of $\Delta_\mu f = \varphi, \ \int_{\R^d} f\d\mu = 0$. We rewrite the term $\langle \operatorname{div}(\rho( w^\varepsilon \otimes w^\varepsilon )), (\nabla f,0)\rangle,$ as \begin{equation} \label{Fisher term weak* conv}
    \frac{\varepsilon^2}{4}\int_0^1\int_{\R^d \times \R^d}\langle\nabla^2 f_t \cdot \nabla_x\log\frac{\d \rho_t^\varepsilon}{\d \mu \otimes\nu}, \nabla_x\log\frac{\d \rho_t^\varepsilon}{\d \mu \otimes\nu}\rangle\d\rho^\varepsilon_t\d t,
\end{equation} 
which we can bound by \begin{equation*}
    \sup_{t\in[0,1]}\lVert\nabla^2 f_t\rVert_{\infty}[ \mathcal{F}_{\varepsilon}(0,0)- \mathcal{A}_0(\rho^\varepsilon) ].
\end{equation*} By the $\Gamma$-convergence result in Theorem \ref{Existence of geodesics} we have that $\mathcal{F}_{\varepsilon}(0,0) \to \mathcal{F}_0(0,0)$. Furthermore, since $\{\mathcal{A}_0(\rho^\varepsilon)\}_{\varepsilon \in [0,1]}$ is bounded, following the proof of \cite[Proposition 3.1]{LouisLacker} we see that, up to a subsequence $\{\varepsilon_k\}_{k\geq1} \downarrow 0$,  $\rho^{\varepsilon_k}_t \to \rho_t$ in the pointwise-in-time topology of weak convergence of measures, where $(\rho_t)_{t \in [0,1]}$ is a minimizer for $\mathcal{F}_0(0,0)$: this can be deduced using the lower semicontinuity of the kinetic energy $\mathcal{A}_0$ and the fact that $\mathcal{F}_{\varepsilon}(0,0) \to \mathcal{F}_0(0,0)$. Moreover we have \begin{equation*}
    0\leq \limsup_{k \to +\infty}\mathcal{F}_{\varepsilon_k}(0,0)- \mathcal{A}_{0}(\rho^{\varepsilon_k}) \leq \mathcal{F}_0(0,0)- \mathcal{A}_0(\rho) = 0,
\end{equation*} so we can conclude the convergence of \eqref{Fisher term weak* conv} to zero without restricting to subsequences.
Now we have to prove that 
\begin{equation} \label{convergence rho v}
    \langle\partial_t(\rho^\varepsilon v^\varepsilon) + \nabla\cdot(\rho^\varepsilon(v^\varepsilon\otimes v^\varepsilon)), (\nabla f,0)\rangle \to \langle \partial_t(\rho^0 v^0) + \nabla\cdot(\rho^0(v^0\otimes v^0)), (\nabla f,0)\rangle,
\end{equation} for some $(\rho^0,v^0)$ that is optimal for the geodesic problem and for every $\varphi \in \tilde{S}_\mu$. By the uniform boundedness of $\mathcal{A}_0(\rho^\varepsilon,v^\varepsilon)$, there exists $(\rho,v)$ that is optimal for the geodesic problem such that $(\rho^\varepsilon, v^\varepsilon) \to (\rho,v)$ with respect to the weak convergence, up to subsequences. 
By readily adapting \cite[Proof of Theorem 4.3.2]{brenier2020examples}, we deduce that \eqref{convergence rho v} holds up to subsequences for every $\varphi \in S_\mu$ such that $\partial_t\nabla f$ is jointly continuous, which holds here using the additional regularity required on $\varphi$. In the end, we do not need to restrict to subsequences for the convergence to hold by the uniqueness of the limit.

\end{proof} 
\end{prop}

\addcontentsline{toc}{section}{Appendix}
\appendix
\section{Appendix: entropy minimization on path space} \label{Appendix}
\setcounter{section}{1}

This appendix informally sketches the connection between the regularized problem \eqref{def regularization} and 
its entropic minimization on path space formulation, which can be the thought of as the Schrödinger bridge problem restricted to the space of couplings $\Pi(\mu,\nu)$.

Let $\rho_0, \ \rho_1 \in \Pi(\mu,\nu)$ be fixed endpoints.
Let $\Omega := C([0,1],\R^d\times\R^d)$ denote the canonical space -- or \emph{path space} --, and let $(X_t)_{t \in [0,1]}$ denote the canonical process on $\Omega$, so that the marginal law $P_t$ at time $t$ of any $P \in \mathcal{P}(\Omega)$ is given by $P_t = (X_t)_{\#}P$. 
We define the measure $R^\varepsilon$ on $\Omega$ as the path law of the diffusion process in $\R^d \times \R^d$,
\[
\d Z_t^{\varepsilon} = \frac{\varepsilon}{2}\nabla\log\mu\otimes\nu(Z_t^{\varepsilon})\d t + \sqrt{\varepsilon}\d B_t, \quad Z_0^{\varepsilon} \sim \mu\otimes\nu, \]
for a Brownian motion $(B_t)_{t \geq 0}$.
Then, the entropic formulation of \eqref{def regularization} is the minimization problem
\begin{equation} \label{entropic pathwise}
    \inf \{ \varepsilon H(P|R^\varepsilon), \ P \in \mathcal{P}(\Omega), \ (P_0,P_1) = (\rho_0,\rho_1), \ P_t \in \Pi(\mu,\nu) \ \forall t \in (0,1)\},
\end{equation} 
where we recall that the pathwise entropy is defined as
\begin{equation*}
   H(P|R^\varepsilon):= \begin{cases}
       \displaystyle\int_{\Omega}\log\frac{\d P}{\d R^\varepsilon}\d P, \quad &\text{if} \ P \ll R^\varepsilon, \\
       +\infty \quad &\text{otherwise}.
   \end{cases}
\end{equation*} 
The precise connection between both formulations is informally stated below.

\begin{prop}[Informal]
    Let $\rho_0, \rho_1 \in \Pi(\mu,\nu)$ be such that $H(\rho_0|\mu\otimes\nu), \ H(\rho_1|\mu\otimes\nu) < +\infty$. Then \eqref{def regularization} and \eqref{entropic pathwise} are equivalent after minimizing over
path measures. More precisely, for any Markovian path measure
$P\in\mathcal P(\Omega)$ such that  $(P_0,P_1) = (\rho_0,\rho_1)$ and $P_t \in \Pi(\mu,\nu) \ \forall t \in [0,1]$, we have 
\begin{equation*}
    \varepsilon H(P|R^\varepsilon)
    \geq
    \mathcal A_\varepsilon(P_\cdot)
    +\varepsilon
    \frac{
    H(\rho_0|\mu\otimes\nu)+H(\rho_1|\mu\otimes\nu)
    }{2}.
\end{equation*}
    where $(P_\cdot ) = (P_t )_{0 \leq t \leq 1}$ is the curve of the time-marginals of $P$, and $\mathcal{A}_{\varepsilon}$ is introduced in Definition \ref{Def reg functional and reg prob}.
    Under the usual finite-entropy hypotheses, equality is recovered after
minimizing over path measures with prescribed marginal curve.
\end{prop}

This statement can be made rigorous using the (probabilistic) Girsanov theory under finite entropy condition \cite{leonard2012girsanov} -- although this goes beyond the scope of the present work.
We refer to \cite{leonard2013survey}
for a rigorous discussion of such pathwise formulations in the context of the Schr{\"o}dinger bridge problem, and to \cite{EntropicInt} and \cite[Section 3]{Baradatoptimality} in the context of the Br{\"o}dinger problem.
We finally emphasize the convenience of the entropic formulation \eqref{entropic pathwise}, which directly benefits from the nice properties of the relative entropy -- such as lower semicontinuity, strict convexity and compactness of level sets --, although \eqref{def regularization} was more adequate to derive a complete PDE description of optimality conditions.

\subsection*{Acknowledgements}

The authors thank Giovanni Conforti and Daniel Lacker for suggesting the problem and fruitful discussions, as well as Maria Colombo and Thomas Gallou{\"e}t.
The authors further thank Roberto Colombo for useful comments and proofreading.

\bibliography{ReportRefs.bib}
\bibliographystyle{alpha}
\nocite{*}

\end{document}